\documentclass[bj]{imsart}

\RequirePackage[OT1]{fontenc}
\RequirePackage{amsthm,amsmath,natbib}
\RequirePackage[colorlinks,citecolor=blue,urlcolor=blue]{hyperref}

\usepackage{a4wide}
\usepackage[english]{babel}

\usepackage{dsfont}
\usepackage{mathrsfs}
\usepackage{amssymb}
\usepackage{verbatim}
\usepackage{graphicx}
\usepackage{xcolor}
\usepackage{fancybox}
\usepackage{enumitem}
\usepackage{eurosym}
\usepackage{comment}

\def\a{\alpha}

\def\g{\gamma}

\def\d{\delta}

\def\t{\theta}

\def\L{\Lambda}
\def\s{\sigma}

\def\vt{\vartheta}

\def\ie{\textit{i.e., }}

\def\cf{\textit{cf. }}
\def\PP{\mathbb P}

\def\RR{\mathbb R}

\def\fcar{\mathds{1}}

\def\esp{\mathbf E}
\def\var{\mathbf{Var}}

\def\prob{\mathbf P}
\def\calN{\mathcal N}

\def\nzeroun{\mathcal{N}(0,1)}

\def\ie{\textit{i.e.}, }

\def\cf{\textit{cf. }}

\def\RR{\mathbb R}

\def\fcar{\mathds{1}}

\def\esp{\mathbf E}
\def\var{\mathbf{Var}}

\def\prob{\mathbf P}
\def\calN{\mathcal N}

\def\nzeroun{\mathcal{N}(0,1)}

\theoremstyle{plain}
\newtheorem{theorem}{Theorem}
\newtheorem{lemma}{Lemma}
\newtheorem{proposition}{Proposition}
\newtheorem{corollary}{Corollary}

\newtheorem*{theorem*}{Theorem}
\newtheorem*{proposition*}{Proposition}
\newtheorem*{corollary*}{Corollary}
\theoremstyle{remark}

\newtheorem*{remark*}{Remark}

\newtheorem*{note*}{Note}
\theoremstyle{definition}

\newtheorem*{definition*}{Definition}

\usepackage{comment}

\begin{document}

\begin{frontmatter}
\title{On estimation of nonsmooth functionals of sparse normal means}
\runtitle{Nonsmooth functional estimation}

\begin{aug}
\author{\snm{O.} \fnms{Collier}\thanksref{a,c,e1}\ead[label=e1,mark]{olivier.collier@parisnanterre.fr}},
\author{\snm{L.} \fnms{Comminges}\thanksref{b,c,e2}
\ead[label=e2,mark]{comminges@ceremade.dauphine.fr}}
\and
\author{\snm{A.B.} \fnms{Tsybakov}\thanksref{c,e3}%
\ead[label=e3,mark]{alexandre.tsybakov@ensae.fr}}

\address[a]{MODAL'X, Universit\'e Paris-Nanterre,
\printead{e1}}

\address[b]{CEREMADE, Universit\'e Paris-Dauphine,
\printead{e2}}

\address[c]{CREST,ENSAE,
\printead{e3}}

\runauthor{Collier, Comminges and Tsybakov}


\end{aug}

\begin{abstract}
\quad We  study the problem of estimation of $N_\g(\t) = \sum_{i=1}^d |\t_i|^\g$ for $\g>0$ and of the $\ell_\g$-norm of $\theta$ for $\g\ge1$ based on the observations $y_i = \t_i + \varepsilon \xi_i, \ i=1,\ldots,d$, where $\t=(\t_1, \dots, \t_d)$ are unknown parameters, $\varepsilon>0$ is known, and  $\xi_i$ are i.i.d. standard normal random variables. We find the non-asymptotic minimax rate for estimation of these functionals on the class of $s$-sparse vectors $\t$ and we propose estimators achieving this rate.

\end{abstract}


\end{frontmatter}

\section{Introduction}\label{sec:introduction}

In recent years, there has been a growing interest in statistical estimation of non-smooth functionals (\cf \cite{CaiLow2011,JiaoVenkatHanWeissman2015,WuYang2015,WuYang2016,HanJiaoMukherjeeWeissman2017,HanJiaoWeissmanWu2017,CarpentierVerzelen2017,FukuchiSakuma2017}).
Some of these papers deal with the normal means model (\cf \cite{CaiLow2011,CarpentierVerzelen2017}) addressing the problems of estimation of the $\ell_1$-norm and of the sparsity index, respectively. In the present paper, we  
analyze a family of nonsmooth functionals including, in particular, the $\ell_1$-norm. We establish non-asymptotic minimax optimal rates of estimation on the classes of sparse vectors and we construct estimators achieving these rates.

Assume that we observe
\begin{equation}\label{model}
y_i = \t_i + \varepsilon \xi_i, \quad i=1,\ldots,d,
\end{equation}
where $\t=(\t_1, \dots, \t_d)$ is an unknown vector of parameters, $\varepsilon>0$ is a known noise level, and  
$\xi_i$ are i.i.d. standard normal random variables. We consider the problem of estimating  the functionals
\begin{equation*}
N_\g(\t) = \sum_{i=1}^d |\t_i|^\g, \quad \g>0, \quad  \text{and}\quad  \|\t\|_\g= \Big(\sum_{i=1}^d |\t_i|^\g\Big)^{1/\g}, \quad \g\ge 1,
\end{equation*}
assuming that the vector $\t$ is $s$-sparse, that is, $\t$ belongs to the class
$$
B_0(s)=\{ \t\in \RR^d : \|\t\|_0\le s\}.
$$ 
Here, $\|\t\|_0$ denotes the number of nonzero components of $\t$ and $s\in \{1,\dots,d\}$.

Set $n_\g(\t) =N_\g(\t)$  for $0<\g\le 1$ and $n_\g(\t) =\|\t\|_\g$  for $\g> 1$.
We measure the accuracy of an estimator $\hat{T}$ of $N_\g(\t)$ by the maximal quadratic risk over $B_0(s)$:
\begin{equation*}
\sup_{\t\in B_0(s)} \esp_\t\big[\big( \hat{T} - n_\g(\t) \big)^2\big].
\end{equation*}
Here and in the sequel, we denote by $\esp_\t$ the expectation with respect to the joint distribution $\prob_\t$
of $(y_1,\dots, y_d)$ satisfying \eqref{model}. 

In this paper, we propose rate-optimal estimators in a non-asymptotic minimax sense, that is, estimators 
${\hat T}_\g^*$ such that
\begin{equation*}
\sup_{\t\in B_0(s) } \esp_\t\big[\big( \hat{T}_\g^* - n_\g(\t) \big)^2\big] \asymp
	\inf_{\hat{T}}\sup_{\t\in B_0(s)} \esp_\t\big[\big( \hat{T} - n_\g(\t) \big)^2\big] := {\mathcal R}_{s,d}(\varepsilon,\g),
\end{equation*}
where $\inf_{\hat{T}}$ denotes the infimum over all estimators and, for two quantities $a$ and $b$ 
possibly depending on $s,d,\varepsilon, \g$, we write $a\asymp b$ if there exist positive
constants $c', c''$ that may depend only on $\g$ such that $c'\le a/b\le c''$.  {\color{black}We establish the following explicit
non-asymptotic  characterization of the minimax risk: 
\begin{equation}\label{optrate0}
	{\mathcal R}_{s,d}(\varepsilon,\g) \asymp
	\begin{cases}
	\ \varepsilon^{2\g}s^{2}\log^\g(1+d/s^2) ,\quad &\text{ if } s\le \sqrt{d} \text{ and }  0<\g\le 1,\\
			\ \varepsilon^{2\g}s^{2}{\log^{-\g}(1+s^2/d)} ,\quad &\text{ if } s> \sqrt{d} \text{ and } 0<\g \le 1,\\		
		\end{cases}
\end{equation}
and
\begin{equation}\label{optrate}
	{\mathcal R}_{s,d}(\varepsilon,\g) \asymp
	\begin{cases}
	\ \varepsilon^{2}s^{2/\g}\log(1+d/s^2) ,\quad &\text{ if } s\le \sqrt{d} \text{ and }  \g>1,\\
			\ \varepsilon^{2}d^{1/\g} ,\quad &\text{ if } s> \sqrt{d} \text{ and } \g \in E,		
		\end{cases}
\end{equation}
where $E$ is the set of all even integers. We also prove that, in the remaining case $s> \sqrt{d}$ and $\g>1$ such that $\g\not \in E$, we have
\begin{equation}\label{optrate1}
c\varepsilon^{2}s^{2/\g}{\log^{1-2\g}(1+s^2/d)}\le {\mathcal R}_{s,d}(\varepsilon,\g)\le \bar c \varepsilon^{2}s^{2/\g}{\log^{-1}(1+s^2/d)}  
\end{equation}   
for some positive constants $c,\bar c$.}

The case $s=d$, $\g=\varepsilon=1$ was studied in \citet*{CaiLow2011}, where it was proved that 
\begin{equation*}\label{rate_cai}
	{\mathcal R}_{d,d}(1,1) = \inf_{\hat{T}}\sup_{\t\in \RR^d} \esp_\t\big[\big( \hat{T} - N_1(\t) \big)^2\big] \asymp \frac{d^2}{\log d}.
\end{equation*}
It was also claimed in  \citet*{CaiLow2011} that ${\mathcal R}_{s,d}(1,1)\asymp s^2/(\log d)$ for $s\ge d^{\beta}$ with $\beta>1/2$, which agrees with the corresponding special case of \eqref{optrate0}. 

We see from \eqref{optrate0} and \eqref{optrate} that, for the general sparsity classes $B_0(s)$ there exist two different regimes with an elbow at $s\asymp \sqrt{d}$. We call them the sparse zone and the dense zone. The estimation methods for these two regimes are quite different. In the sparse zone, where $s$ is smaller than $\sqrt{d}$, we show that one can use suitably adjusted thresholding to achieve optimality. In this zone, rate optimal estimators can be obtained based on  the techniques developed in  \cite{CCT2017} to construct minimax optimal estimators of linear and quadratic functionals. In the dense zone, where $s$ is greater than $\sqrt{d}$, we use another approach. We follow the general scheme of estimation of non-smooth functionals from \cite{LepskiNemirovskiSpokoiny1999} and our construction is especially close in the spirit to \cite{CaiLow2011}. Specifically, we consider the best polynomial approximation of the function $|x|^{\g}$ in a neighborhood of the origin and plug in unbiased estimators for each power in the expression of this polynomial.  Outside of this neighborhood, for $i$ such that $|y_i|$ is, roughly speaking, greater than the "noise level" of the order of $\varepsilon\sqrt{\log d}$, we use $|y_i|^\g$ as an estimator of $|\t_i|^\g$.
The main difference from the estimator suggested in~\cite{CaiLow2011} for $\g=1$ lies in the fact that, for the polynomial approximation part, we need to introduce a block structure with exponentially increasing blocks and carefully chosen thresholds depending on~$s$. This is needed to achieve optimal bounds for all $s$ in the dense zone and not only for $s=d$ or $s$ comfortably greater than $\sqrt{d}$ as in ~\cite{CaiLow2011}.

%

In the present work, the variance $\varepsilon^2$ of the noise and the sparsity parameter $s$ need to be known exactly. 
We conjecture that adaptation to $\varepsilon^2$ can be done without loss of the rate in the sparse zone $s\le \sqrt{d}$. In the dense zone, the optimal rate can deteriorate dramatically when $\varepsilon^2$ is unknown as shown in~\cite{CCT2018} for the case $\g=2$.
This contrasts with the results for linear functionals. Indeed, in~\cite{CCTV2018} it is proved that, for linear functionals, adaptation to $\varepsilon^2$ can be done without loss in the rate, and adaptation to $s$ only brings a logarithmically small deterioration of the rate. 

This paper is organized as follows. In Section \ref{section_upperbounds}, we introduce the estimators and state the upper bounds for their risks. Section~\ref{sec:lower} provides the matching lower bounds. The rest of the paper is devoted to the proofs.  In particular, some useful results from approximation theory are collected in Section \ref{sec:approximation}.

\section{Definition of estimators and upper bounds for their risks}\label{section_upperbounds}

In this section, we propose two different estimators, for the dense and sparse regimes defined by the inequalities $s^2\ge 4d$ and $s^2< 4d$, respectively. Recall that, in the Introduction, we used the inequalities $s\ge \sqrt{d}$ and $s<  \sqrt{d}$, respectively, to define the two regimes. The factor~4 that we introduce in the definition here is a matter of convenience for the proofs. We note that such a change does not influence the final 
result since the optimal rate (cf. \eqref{optrate}) is the same, up to a constant, for all $s$ such that $s\asymp  \sqrt{d}$.

\subsection{Dense zone: $s^2\ge 4d$}

{\color{black} We first study the problem of estimation of $n_\g(\t)$ in the dense zone. Two estimators will be proposed -- the first one that achieves optimality when $\g$ is not an even integer, and the second one for even integer $\gamma$. They are derived from two estimators of $N_\g(\t)$ that we are going to define now.

	We first present the estimator of $N_\g(\t)$ that will be used when $\g$ is not an even integer.} For any positive integer $K$, we denote by $P_{\g,K}(\cdot)$ the best approximation of $|x|^\g$ by polynomials of degree at most $2K$ on the interval $[-1,1]$, that is 
\begin{equation*}\label{}
 \max_{x\in [-1,1]} \Big||x|^\g-P_{\g,K}(x)\Big|= \min_{G\in \mathcal{P}_{2K}} \max_{x\in [-1,1]} \Big| |x|^\g-G(x)\Big|,
\end{equation*}
where $\mathcal{P}_K$ is the class of all real polynomials of degree at most $K$. 	
Since $|x|^\g$ is an even function, it suffices to consider approximation by polynomials of even degree.
	The quality of the best polynomial approximation of $|x|^\g$ is described by Lemma~\ref{approx} in Section~\ref{sec:approximation}.
	
		We denote by $a_{\g,2k}$ the coefficients of the canonical representation of $P_{\g,K}$:  
	\begin{equation*}\label{}
		 P_{\g,K}(x) = \sum_{k=0}^K a_{\g,2k} x^{2k}, \quad x \in \RR,
	\end{equation*}
	and by $H_k(\cdot)$ the $k$th Hermite polynomial
	\begin{equation*}\label{definition_Hermite}
		\ H_k(x) = (-1)^k e^{x^2/2}\frac{d^k}{d x^k} e^{-x^2/2}, \quad k\in{\mathbb N}, \quad x\in \RR.
	\end{equation*}
	To construct our first estimator in the dense zone, we use the sample duplication device, \ie 
	we transform $y_i$ into randomized observations $y_{1,i},y_{2,i}$ as follows.
Let  $z_1,\dots, z_d$ be i.i.d. random variables such that $z_i\sim \mathcal{N}(0,\varepsilon^2)$ and $z_1,\dots, z_d$ are independent of $y_1,\dots,y_d$.  Set
	\begin{equation*}\label{duplicatedsample}
	y_{1,i} = y_i + z_i, \qquad
		\ y_{2,i} = y_i - z_i, \qquad  i=1,\dots,d.
	\end{equation*}
		Then, $y_{1,i} \sim \mathcal{N}(\t_i,\s^2)$, $y_{2,i} \sim \mathcal{N}(\t_i,\s^2)$ for $i=1,\dots, d,$
where $\s^2=2\varepsilon^2$ and the random variables $(y_{1,1},\dots, y_{1,d}, y_{2,1},\dots, y_{2,d})$ are mutually independent.
		 
		 Define the estimator of $N_\g$ as follows:
	\begin{equation}\label{definition_estimateurgamma}
		\hat{N}_\g = \sum_{i=1}^d \xi_\g(y_{1,i},y_{2,i})
	\end{equation}
where $$\xi_\g(u,v) = \sum_{l=0}^L\hat{P}_{\g,K_l,M_l}(u) \fcar_{\s t_{l-1}<|v|\le\s t_l} + |u|^\g \fcar_{|v|>\s t_L},
$$ 
and
	\begin{equation}\label{parametres}
		\begin{cases}
			\ \hat{P}_{\g,K,M}(u) = \sum_{k=1}^K \s^{2k} a_{\g,2k} M^{\g-2k} H_{2k}(u/\s), \phantom{\Big()} \\
			\ K_l = 4^lc\log(s^2/d), \phantom{\Big()} \\
			\ M_l = 2^{l+1}\s\sqrt{2\log(s^2/d)}, \phantom{\Big()} \\
			\ t_l = 2^l \sqrt{2\log(s^2/d)},\, t_{-1}=0, \phantom{\Big()} \\
			\ L \text{ is the smallest integer such that } 2^L \ge {3}\sqrt{\log(d)/\log(s^2/d)}. \phantom{\Big()}
		\end{cases}
	\end{equation}	
	Here and in what follows $\fcar_{\{\cdot\}}$ denotes the indicator function, and $c>0$ is a constant that will be chosen small enough (see the proof of Theorem~\ref{theorem_upperbound1} below).


The next theorem provides an upper bound on the risk of $\hat{N}_\g$ as estimator of $N_\g(\t)$ in the dense zone.

\begin{theorem}\label{theorem_upperbound1}
	 Let the integers $d$ and $s\in \{1,\dots,d\}$ be such that $s^2\ge4d$ and let $\g>0$. Then for any $\t\in B_0(s)$ the estimator defined in~(\ref{definition_estimateurgamma}) satisfies
{\color{black}	\begin{equation*}
		\esp_\t \big[\big(\hat{N}_\g-N_\g(\t)\big)^2\big] \le C\left(\frac{\varepsilon^{2\g}s^2}{\log^\g(s^2/d)} + \frac{ \varepsilon^2 s^{2/\g}}{\log(s^2/d)} \|\t\|_{\g}^{2\g-2}
		 \fcar_{\g>1}
		\right)  
	\end{equation*}}
	where $C>0$ is a constant depending only on $\g$. 
\end{theorem}

{\color{black}
Inspection of the proof in Section \ref{sec:proof:upper} shows that the block structure of the estimator is needed to retrieve the sharp logarithmic factor $\log^\g(s^2/d)$ in the rate for all $s\ge 2\sqrt{d}$. If $s$ is substantially greater than $\sqrt{d}$, for example,  $d^a<s<d$ for some $a>1/2$, then this logarithmic factor is equivalent to $\log^\g(d)$, and it is enough to use the estimator with only two blocks in order to obtain the result.

 Although Theorem~\ref{theorem_upperbound1} is valid for all $\g>0$ it will be useful for us only when $\g$ is not an even integer since there exist estimators achieving better rates in the dense zone for even integers $\g$. We now provide a construction of such an estimator.

Indeed, assume more generally that $\g$ is an integer, not necessarily even.
 We now use the sample "cloning" device as above but instead of creating two independent randomized samples,
 we create~$\g$ independent randomized samples $(y_{i,m}, 1,\dots, d), \ m=1,\dots,\g,$  with variance multiplied by $\g$:
    $$y_{i,m}=\t_i+\varepsilon \sqrt{\g}\,\xi_{i,m}, \quad \g,$$ 
    where $\xi_{i,m}$ are i.i.d. standard normal random variables (see \cite{Nemirovski2000} for details).
    
We can now estimate the value $\sum_{i=1}^d  \t_i^{\g}$ by
    \begin{equation}\label{tildeN}
    \tilde{N}_\g = \sum_{i=1}^d \prod_{m=1}^{\g}y_{i,m}.
    \end{equation}  
Since $\esp \Big( \prod_{m=1}^{\g}y_{i,m}\Big)= \t_i^{\g}$ we find immediately that $\tilde{N}_\g$ is an unbiased estimator of $\sum_{i=1}^d  \t_i^{\g}$: 
    $$\esp \Big( \sum_{i=1}^d \prod_{m=1}^{\g}y_{i,m}\Big)= \sum_{i=1}^d  \t_i^{\g}.$$
If $\g$ is an even integer, 
	$\sum_{i=1}^d  \t_i^{\g} =N_\g(\t).$     
    The risk of $\tilde{N}_\g$ admits the following bound.

\begin{theorem}\label{theorem_upperbound_even}
	 Let $\g$ be an integer. Then, for any $\t\in \RR^d$ we have 
	\begin{equation*}
		 \esp_\t \big[\big(\tilde{N}_\g-\sum_{i=1}^d  \t_i^{\g}\big)^2\big] \le C\left(\varepsilon^{2\g}d + \varepsilon^2 \|\t\|_{2\g-2}^{2\g-2}\right)  
	\end{equation*}
	where $C>0$ is a constant depending only on $\g$. In particular, if $\g$ is an even integer we have here 
	$\sum_{i=1}^d  \t_i^{\g} =N_\g(\t).$ 
\end{theorem}

Note that this theorem is valid for any sparsity $s$ but we will use it only in the dense zone since in the sparse zone there exist better 
estimators achieving the optimal rate, cf. Section \ref{sec:sparse_upper} below.

As a consequence of Theorems \ref{theorem_upperbound1} and \ref{theorem_upperbound_even}, we obtain the following result for estimation of the norm $\|\t\|_{\g}$.

\begin{theorem}\label{corollary_upperbound1}
	(i) Let the integers $d$ and $s\in \{1,\dots,d\}$ be such that $s^2\ge4d$ and let $\g>1$. Set   
	$\hat{n}_\g = |\hat{N}_\g|^{1/\g}$, where $\hat{N}_\g$ is defined in~(\ref{definition_estimateurgamma}). Then
	\begin{equation}\label{upper_dense}
		\sup_{\t\in B_0(s)} \esp_\t \big[\big(\hat{n}_\g-\|\t\|_{\g} \big)^2\big] \le C\frac{\varepsilon^{2}s^{2/\g}}{\log(s^2/d)},
	\end{equation}
	where $C>0$ is a constant depending only on $\g$.
	
	(ii)  Let $\g$ be an even integer. Set $\tilde{n}_\g = |\tilde{N}_\g|^{1/\g}$, where $\tilde{N}_\g$ is defined in \eqref{tildeN}. Then
	\begin{equation}\label{upper_even}
		 \sup_{\t\in \RR^d} \esp_\t \big[\big(\tilde{n}_\g- \|\t\|_{\g}\big)^2\big] \le C \varepsilon^{2}d^{1/\g},
		 \end{equation}
	where $C>0$ is a constant depending only on $\g$. 
\end{theorem}

We will prove below that the second bound of Theorem~\ref{corollary_upperbound1} cannot be improved in a minimax sense, and that the first is optimal up to a possible logarithmic factor. Note that, in the dense zone $s^2\ge4d$ considered in Theorem~\ref{corollary_upperbound1}(i), the right hand side of~\eqref{upper_even}  is of smaller order than the right hand side of~\eqref{upper_dense}. This privileged position of even powers~$\g$  can be explained by the fact that the even power functionals $N_\g(\t)$ admit unbiased estimators converging with much faster rates than the estimators for other values of $\g$, for which the functionals $N_\g(\t)$ are not smooth.

}

\subsection{Sparse zone: $s^2\le4d$}\label{sec:sparse_upper}

If $s$ belongs to the sparse zone $s^2\le4d$ we use the estimator
\begin{equation}\label{definition_estimateursparse}
	\hat{N}_\g^* = \sum_{i=1}^d \big\{|y_i|^\g-\varepsilon^\g\a_\g\big\}\fcar_{y_i^2>2\varepsilon^2\log(1+d/s^2)},
\end{equation}
where
\begin{equation*}\label{definition_alpha}
	\a_\g = \frac{\esp\big(|\xi|^\g\fcar_{\xi^2>2\log(1+d/s^2)}\big)}{\prob\big(\xi^2>2\log(1+d/s^2)\big)} \quad \text{for} \quad \xi\sim\nzeroun.
\end{equation*}
The next theorem establishes an upper bound on the risk of this estimator.
\begin{theorem}\label{theorem_upperbound2}
Let the integers $d$ and $s\in \{1,\dots,d\}$ be such that
	$s^2\le4d$ and $\g>0$. 
	Then for any $\t\in B_0(s)$ the estimator defined in~(\ref{definition_estimateursparse}) satisfies
	{\color{black}\begin{equation*}
	\esp_\t \big[\big(\hat{N}_\g^*-N_\g(\t)\big)^2\big] \le 
	C\left(\varepsilon^{2\g}s^2\log^\g(1+d/s^2) + \varepsilon^2 s^{2/\g} \|\t\|_{\g}^{2\g-2}
	\fcar_{\g>1}\right)  ,
	\end{equation*}}
	where $C>0$ is a constant depending only on $\g$.
\end{theorem}

	Note that the  estimator $\hat{N}_\g^*$ can be viewed as an example of applying the following routine developed in~\cite{CCT2017}. We start from the direct estimator $ \sum_{i=1}^d |y_i|^\g$ and  we threshold every term in this sum.  This estimator being biased, we center every term by its mean under the assumption that there is no signal.  
	Finally, we choose the value of the threshold that makes the best compromise between the first and second type errors in the support estimation problem. As opposed to the dense zone, we do not invoke the polynomial approximation.
	  In fact, one can notice that the polynomial approximation is only useful in a neighborhood of $0$ but in the sparse zone we renounce estimating small instances of $\t_i$.

{\color{black}Finally, we derive a consequence of Theorem \ref{theorem_upperbound2} for estimation of the functional $\|\t\|_{\g}$.

\begin{theorem}\label{corollary_upperbound2}
	Let the integers $d$ and $s\in \{1,\dots,d\}$ be such that $s^2\le4d$ and $\g>1$. 
	Set  
	$\hat{n}_\g^* = |\hat{N}_\g^*|^{1/\g}$, where $\hat{N}_\g^*$ is defined in~(\ref{definition_estimateursparse}). Then
	\begin{equation*}
		\sup_{\t\in B_0(s)} \esp_\t \big[\big(\hat{n}_\g^*-\|\t\|_{\g} \big)^2\big] \le C \varepsilon^{2}s^{2/\g}\log(1+d/s^2),
	\end{equation*}
	where $C>0$ is a constant depending only on $\g$.
\end{theorem}
}

\section{Lower bounds}\label{sec:lower}
  
 We denote by $\mathcal L$ the set of all monotone non-decreasing functions $\ell:[0, \infty)\to [0, \infty)$ such that
$\ell(0)=0$ and $\ell\not\equiv 0$.
  
\begin{theorem} \label{theo_inf_sparse} Assume that $\g>0$. Let $s,d$ be integers such that $s\in \{1,\dots,d\}$ and 
let $\ell(\cdot)$ be any loss function in the class $\mathcal L$. There exist  positive constants $c_1$ and  $c_2$  depending only on $\g$ and $\ell(\cdot)$ such that, 
{\color{black}
for $\phi=\varepsilon^{\g} s\log^{\frac{\g}{2}}(1+ {d}/{s^2})$, 
$$\inf_{\hat{T}} \sup_{\t \in B_0(s)\cap \{\|\t\|_\g\le \phi^{1/\g}\}} \esp_\t\, \ell\Big(  c_1\phi ^{-1}|\hat{T}-N_\g(\t)|\Big) \ge c_2,$$
where $\inf_{\hat{T}}$ denotes the infimum over all estimators. 	
}
\end{theorem}
The proof is omitted since it follows the lines of the proof of the lower bound in~\cite[Theorem 1]{CCT2017} with the only difference that $L(\t)=\sum_{i=1}^d \t_i$ should be replaced by $\sum_{i=1}^d \t_i^\g$. 
{\color{black}The fact that the result is valid not only over $B_0(s)$ but also over the intersection of $B_0(s)$ with $B_\g:=\{\|\t\|_\g\le \phi^{1/\g}\}$ is granted since  the support of the prior measure used in the proof of the lower bound in~\cite[Theorem 1]{CCT2017}  is included in $B_\g$ for any $\g>0$.

As a corollary of Theorem~\ref{theo_inf_sparse} we obtain the following lower bound.

\begin{theorem}\label{theorem_lowerbound_norm1}
	Assume that $\g>1$. Let $s,d$ be integers such that $s\in \{1,\dots,d\}$ and 
let $\ell(\cdot)$ be any loss function in the class $\mathcal L$. There exist  positive constants 
$c_1$ and  $c_2$  depending only on $\g$ and $\ell(\cdot)$ such that
$$\inf_{\hat{T}} \sup_{\t \in B_0(s)} \esp_\t\, \ell\Big(  c_1( \varepsilon s^{1/\g} \log^{1/2}(1+ {d}/{s^2}))^{-1}|\hat{T}-\|\t\|_\g|\Big) \ge c_2,$$
where $\inf_{\hat{T}}$ denotes the infimum over all estimators. 
\end{theorem}

Although Theorems~\ref{theo_inf_sparse} and~\ref{theorem_lowerbound_norm1} are valid with no restriction on $s\in \{1,\dots,d\}$, they yield suboptimal bounds in the dense zone. We now turn to the minimax lower bounds with better rates in the dense zone. We state them in the next three theorems of this section.

}

\begin{theorem}
\label{theo_inf} 
Assume that $0<\g\le 1$. Let $s,d$ be integers such that $s\in \{1,\dots,d\}$ and  let $\ell(\cdot)$ be any loss function in the class $\mathcal L$.
{\color{black}
There exist  positive constants $c_1$, $c_2$ and  $c_3$  depending only on $\g$  and $\ell(\cdot)$ and a constant $\bar C\ge 4$ depending only on $\g$ such that, if  $s^2\ge \bar C d$ and $\phi=c_3\varepsilon^{\g} s\log^{-\frac{\g}{2}}(s^2/d)$, then 
$$\inf_{\hat{T}} \sup_{\t \in B_0(s)} \esp_\t\, \ell\bigg(  c_1 \phi^{-1}|\hat{T}-N_\g(\t)|\bigg) \ge c_2.$$
where  $\inf_{\hat{T}}$ denotes the infimum over all estimators. 
}
\end{theorem}

{\color{black}

We are now in a position to derive the result on the minimax rate for  $0<\g\le 1$ announced in~\eqref{optrate0}. It is not hard to see that it follows from  Theorems~\ref{theorem_upperbound1},~\ref{theorem_upperbound2},~\ref{theo_inf_sparse} and~\ref{theo_inf} with $\ell(u)=u^2$.
 
Next, the minimaxity of rate in the first line of \eqref{optrate} is granted by Theorems~\ref{corollary_upperbound2} and~\ref{theorem_lowerbound_norm1} while the second line of \eqref{optrate}  follows from Theorem~\ref{corollary_upperbound1}(ii) and the next lower bound.

\begin{theorem}
\label{theorem_lowerbound_norm2} 
Assume that $\g$ is an even integer. Let $s,d$ be integers such that $s\in \{1,\dots,d\}$ and $s\ge \sqrt{d}$. Let $\ell(\cdot)$ be any loss function in the class $\mathcal L$.
Then there exist  positive constants $c_1$ and  $c_2$  depending only on $\g$  and $\ell(\cdot)$ such that
\begin{equation}\label{low2}
\inf_{\hat{T}} \sup_{\t \in B_0(s)} \esp_\t\, \ell\bigg(  c_1 \big(
\varepsilon d^{\frac{1}{2\g}}
\big)^{-1}|\hat{T}-\|\t\|_\g|\bigg) \ge c_2
\end{equation}
where  $\inf_{\hat{T}}$ denotes the infimum over all estimators. 	
\end{theorem}

}

{\color{black}

In conclusion, we have the following corollary.

\begin{corollary}
\label{corollary_main} 
The minimax risk on $B_0(s)$ with loss function $\ell(u)=u^2$ satisfies \eqref{optrate0} and \eqref{optrate}.
\end{corollary}

Finally, we deduce \eqref{optrate1} from Theorem~\ref{corollary_upperbound1}(i) and the following lower bound.

\begin{theorem}\label{theorem_lowerbound_norm_hard} 
Assume that $\g>1$ is not an even integer. Let $s,d$ be integers such that $s\in \{1,\dots,d\}$ and  let $\ell(\cdot)$ be any loss function in the class $\mathcal L$.
There exist  positive constants $c_1$ and  $c_2$  depending only on $\g$  and $\ell(\cdot)$ and a constant $\bar C\ge 4$ depending only on $\g$ such that, if  $s^2\ge \bar C d$, then 
\begin{equation}\label{low1}
\inf_{\hat{T}} \sup_{\t \in B_0(s)} \esp_\t\, \ell\bigg(  c_1 \bigg(
\frac{\varepsilon s^{1/\g}}{\log^{\g-1/2}(s^2/d)}
\bigg)^{-1}|\hat{T}-\|\t\|_\g|\bigg) \ge c_2
\end{equation}
where  $\inf_{\hat{T}}$ denotes the infimum over all estimators. 	
\end{theorem}
}

\section{Proofs of the upper bounds}

Throughout the proofs, we denote by $C$ positive constants that can depend only on $\g$ and may take different values on different appearances. 

\subsection{Proof of Theorem~\ref{theorem_upperbound1}}\label{sec:proof:upper}

Denote by $S$ the support of $\t$. We start with a bias-variance decomposition
\begin{align*}
	\big(\hat{N}_\g-N_\g(\t)\big)^2 &\le 4\,\Big(\sum_{i\in S} \esp_\t\xi_\g(y_{1,i},y_{2,i})-\sum_{i\in S} |\t_i|^\g\Big)^2 \\
	&+ 4\,\Big(\sum_{i\in S} \xi_\g(y_{1,i},y_{2,i})-\sum_{i\in S} \esp_\t\xi_\g(y_{1,i},y_{2,i})\Big)^2 \\
	&+ 4\,\Big(\sum_{i\not\in S} \esp_\t \xi_\g(y_{1,i},y_{2,i}) \Big)^2 \\
	&+ 4\,\Big(\sum_{i\not\in S} \xi_\g(y_{1,i},y_{2,i})-\sum_{i\not\in S} \esp_\t\xi_\g(y_{1,i},y_{2,i})\Big)^2
\end{align*}
leading to the bound
\begin{align}\label{proof_eq1}
	\esp_\t\big[(\hat{N}_\g-N_\g(\t))^2\big] &\le 4\Big(\sum_{i\in S} B_i \Big)^2 + 4 \sum_{i\in S} V_i \\
	&+ 4d^2 \max_{i\not\in S} B_i^2 + 4d \max_{i\not\in S} V_i, \nonumber
\end{align}
where $B_i=\esp_\t \xi_\g(y_{1,i},y_{2,i})- |\t_i|^\g$ is the bias of $\xi_\g(y_{1,i},y_{2,i})$ as an estimator of $|\t_i|^\g$ and $V_i=\var_\t(\xi_\g(y_{1,i},y_{2,i}))$ is its variance. We now bound separately the four terms in  \eqref{proof_eq1}.  
{\color{black}
We will show that the first two terms are smaller than  
\begin{equation}\label{210}
C\left({\s^{2\g}s^2}{\log^{-\g}(s^2/d)} + \frac{\s^2}{\log(s^2/d)} \Big(\sum_{i\in S}|\t_i |^{\g-1}\Big)^2\fcar_{\g>1}+ \s^2 \sum_{i=1}^d |\t_i |^{2\g-2}\fcar_{\g>1}\right)  
\end{equation}
while the last two terms are smaller than $C\s^{2\g} s^2 \log^{-\g}(s^2/d)$. This proves the theorem since, by H\"older inequality, for any $\t\in B_0(s)$ and $\g>1$ we have
\begin{equation}\label{21}
\Big(\sum_{i\in S}|\t_i |^{\g-1}\Big)^2\le s^{2/\g}\|\t\|_\g^{2\g-2},   
\end{equation}
and 
\begin{equation}\label{22}
\sum_{i=1}^d |\t_i|^{2\g-2}\le s^{1/\g} \Big(\sum_{i=1}^d |\t_i|^{2\g}\Big)^{\frac{\g-1}{\g}}\le s^{1/\g}\|\t\|_\g^{2\g-2}.
\end{equation}
}
$1^{\circ}.$ {\it Bias for $i\not\in S$.} For $i\not\in S$ using Lemma~\ref{lemma_hermite2} we obtain
\begin{align*}
	|B_i| &= \s^\g \esp |\xi|^\g  \prob(|\xi|>t_L) \le C \s^\g e^{-t_L^2/2}, \quad \xi\sim\nzeroun.
\end{align*}
The last exponential is smaller than $1/d$ by the definition of $t_L$, so that
\begin{equation}\label{proof_eq2}	
	d^2 \max_{i\not\in S} B_i^2 \le C \s^{2\g}d  \le C \frac{\s^{2\g} s^2}{\log^{\g}(s^2/d)}. 
\end{equation}

$2^{\circ}.$ {\it Variance for $i\not\in S$.} If $i\not\in S$, then
\begin{equation}\label{proof_eq1a}
	V_i \le \sum_{l=0}^L \esp \hat{P}^2_{\g,K_l,M_l}(\s \xi) \prob(|\xi|>t_{l-1}) + \s^{2\g} \esp |\xi|^{2\g} \prob(|\xi|>t_L), \quad \xi\sim\nzeroun.
\end{equation}
The last term in \eqref{proof_eq1a} is bounded from above as in item $1^{\circ}.$ Next, in view of Lemma~\ref{lemma_CaiLow1},
$$
\esp \hat{P}^2_{\g,K_0,M_0}(\s\xi) \le C \s^{2\g}\frac{6^{2K_0} }{(M_0/\s)^{4-2\g}}\le C \s^{2\g}\log^\g(s^2/d) \Big(\frac{s^2}{d}\Big)^{2c\log 6} \le C \s^{2\g}\log^\g(s^2/d) \sqrt{\frac{s^2}{d}}
$$
if $c$ is chosen such that $2c\log 6\le 1/2$. Here, we use the assumption $s^2\ge 4d$.
 For $l\ge 1$, we use Lemma~\ref{lemma_CaiLow1} to obtain
\begin{align*}
	\esp \hat{P}^2_{\g,K_l,M_l}(\s\xi) \prob(|\xi|>t_{l-1}) &\le C \s^{2\g}\frac{6^{2K_l} e^{-t^2_{l-1}/2}}{(M_l/\s)^{4-2\g}} \le C \s^{2\g}4^{\g l}\log^\g(s^2/d) \Big(\frac{s^2}{d}\Big)^{(2c\log 6-1/4) 4^l} \\ &\le C \s^{2\g}4^{\g l}\log^\g(s^2/d) \Big(\frac{s^2}{d}\Big)^{-4^{l}/8}
\end{align*}
if we chose $c$ such that $2c\log 6\le 1/8$.  In conclusion, under this choice of $c$, using the fact that $s^2\ge 4d$, we get 
\begin{equation}\label{proof_eq3}
	d \max_{i\not\in S} V_i \le C \s^{2\g}d\log^\g(s^2/d) \sqrt{\frac{s^2}{d}} \le \frac{C \s^{2\g}s^2}{\log^\g(s^2/d)}.
\end{equation}

\medskip

$3^{\circ}.$ {\it Bias for $i\in S$.} If $i\in S$, the bias has the form
	\begin{align*}
		B_i &= \sum_{l=0}^L \esp \hat{P}_{\g,K_l,M_l}(X)\,\prob(\s t_{l-1}<|X|\le\s t_l) 
		+ \esp |X|^\g \, \prob(|X|>\s t_L ) - |\t_i|^\g,
	\end{align*}
	where $X\sim\calN(\t_i,\s^2)$. We will analyze this expression separately in three different ranges of values of $|\t_i|$.

\smallskip

$3.1^{\circ}.$ {\it Case $0<|\t_i|<2\s t_0$. } In this case, we use the bound
	\begin{align*}
		|B_i| &\le \max_l \big|\esp \hat{P}_{\g,K_l,M_l}(X)-|\t_i|^\g\big| + \big|\esp|X|^\g-|\t_i|^\g \big| \, \prob(|X|>\s t_L),
	\end{align*}
	where $X\sim\calN(\t_i,\s^2)$. Since $|\t_i|\le M_l$ for all~$l$, we can use Lemma~\ref{lemma_CaiLow2} to obtain 
	\begin{equation}\label{proof_eq3a}
		\big|\esp \hat{P}_{\g,K_l,M_l}(X)-|\t_i|^\g\big| \le C \Big(\frac{M_l}{K_l}\Big)^\g \le \frac{C\s^\g}{\log^{\g/2}(s^2/d)}.
	\end{equation}
	In addition, using Lemma~\ref{lemma_pluggedestimator} and the inequalities $t_L>3t_0\ge 3|\t_i|/(2\s)$, $3\sqrt{2\log(d)}\le t_L\le 6\sqrt{2\log(d)}$ we get
	\begin{align*}
		\big|\esp|X|^\g-|\t_i|^\g \big| \ \prob(|X|>\s t_L) &\le C (\s^\g + \s^2|\t_i|^{\g-2}\fcar_{|\t_i|>\s})  \, \prob(|\xi|>t_L-|\t_i|/\s)\\
		 &\le C \s^\g (1+(\log^{\g/2}d)\fcar_{\g>2} )\, \prob(|\xi|>t_L/3) \le \frac{C\s^\g}{\log^{\g/2}(s^2/d)}
	\end{align*}
	where $\xi\sim\nzeroun$. It follows that
	\begin{equation}\label{proof_eq4}
		s^2 \max_{i:0<|\t_i|<2\s t_0} B^2_i \le \frac{C\s^{2\g}s^2}{\log^\g(s^2/d)}.
	\end{equation}
\smallskip

$3.2^{\circ}.$ {\it Case $2\s t_0<|\t_i|\le2\s t_L$. }	
	 Let $l_0\in\{1,\ldots,L-1\}$ be the integer such that $\s t_{l_0} < |\t_i| \le \s t_{l_0+1}$.  We have
	\begin{align}\label{proof_eq4a}
		|B_i| &\le \sum_{l=0}^{l_0-1} \big|\esp \hat{P}_{\g,K_l,M_l}(X)-|\t_i|^\g\big|\cdot\prob(\s t_{l-1}<|X|\le\s t_l)  \\
		&+ \max_{l\ge l_0} \big|\esp \hat{P}_{\g,K_l,M_l}(X)-|\t_i|^\g\big| + \big|\esp|X|^\g-|\t_i|^\g \big|,\nonumber
 	\end{align}
	where $X\sim\calN(\t_i,\s^2)$.  Analogously to \eqref{proof_eq3a} we find
	\begin{equation*}\label{}
		 \max_{l\ge l_0}\big|\esp \hat{P}_{\g,K_l,M_l}(X)-|\t_i|^\g\big|  \le \frac{C\s^\g}{\log^{\g/2}(s^2/d)}.
	\end{equation*}
	Next, Lemma~\ref{lemma_pluggedestimator} and the fact that $|\t_i|>2\s t_0=2\s \sqrt{2\log(s^2/d)}$ imply
	{\color{black}
	\begin{align}\label{eqproofo}
		\big|\esp|X|^\g-|\t_i|^\g \big| &\le C\s^2 |\t_i|^{\g-2}\le C\left( \frac{\s^\g }{\log^{\g/2}(s^2/d)}\fcar_{\g\le 1}+
		\frac{\s |\t_i|^{\g-1}}{\sqrt{\log(s^2/d)}}\fcar_{\g>1} \right).
	\end{align}
	}
Finally, we consider the first sum on the right hand side of \eqref{proof_eq4a}. Notice that
	\begin{equation*}
		\prob(\s t_{l-1}<|X|\le\s t_l) \le e^{-\frac{\t_i^2}{8\s^2}}, \quad l=0,\dots,l_0-1,
	\end{equation*}
	since $|\t_i| > \s t_{l_0} \ge 2\s t_l$ for $l<l_0$.	Using these inequalities and Lemma~\ref{lemma_CaiLowamlior} we get 
	\begin{align*}
		\sum_{l=0}^{l_0-1} \big|\esp \hat{P}_{\g,K_l,M_l}(X)\big|\cdot\prob(\s t_{l-1}<|X|\le\s t_l) &\le C \s^\g \sum_{l=0}^{l_0-1} 6^{K_l} K_l^{1+\g/2} e^{(c-1)\t_i^2/(8\s^2)} \\
		&\le C \s^\g \sum_{l=0}^{l_0-1} t_l^{2+\g} e^{(c\log 6+c-1)t_l^2/2}.
	\end{align*}
	Choose $c>0$ such that $c \log 6+c <1/4$. As $t_l = 2^l \sqrt{2\log(s^2/d)}$, this yields
	\begin{align*}
		\sum_{l=0}^{l_0-1} \big|\esp \hat{P}_{\g,K_l,M_l}(X)\big|\cdot\prob(\s t_{l-1}<|X|\le\s t_l)  &\le C \s^\g e^{-(1/2)\log(s^2/d)} \\
		&\le \frac{C\s^\g}{\log^{\g/2}(s^2/d)}.
	\end{align*}
	Furthermore,
	\begin{align}\label{eqproof4ab}
		\sum_{l=0}^{l_0-1} |\t_i|^\g \prob(\s t_{l-1}<|X|\le\s t_l) &\le l_0 |\t_i|^\g e^{-\frac{\t_i^2}{8\s^2}} \\
		&\le C \log\Big(\frac{\t_i^2}{2\s^2\log(s^2/d)}\Big) |\t_i|^\g e^{-\frac{\t_i^2}{8\s^2}} \phantom{\sum_l^l}\nonumber\\
		&\le C \s^\g e^{-\frac{\t_i^2}{16\s^2}}, \phantom{\Big()\sum^l}\nonumber
	\end{align}
	where we have used that $|\t_i|>\s t_{l_0}= \s 2^{l_0} \sqrt{2\log(s^2/d)}$. Since $l_0\ge 1$, this also implies that \eqref{eqproof4ab} does not exceed
		\begin{equation*}
		\frac{C\s^\g}{\log^{\g/2}(s^2/d)}.
	\end{equation*}
	Combining the above arguments yields
	{\color{black}
	\begin{equation}\label{proof_eq5}
		\Big( \sum_{i\in S: 2\s t_0<|\t_i|\le2\s t_L} B_i \Big)^2 \le C\left(\frac{\s^{2\g}s^2}{\log^\g(s^2/d)} + \frac{\s^2}{\log(s^2/d)}\Big( \sum_{i\in S} |\t_i|^{\g-1} \Big)^2\fcar_{\g>1}\right).
	\end{equation}
	}
\vspace{3mm}
$3.3^{\circ}.$ {\it Case $|\t_i|>2\s t_L$. }			
	 Recall that the bias $B_i$ has the form
	\begin{align*}
		B_i &= \sum_{l=0}^L \esp \hat{P}_{\g,K_l,M_l}(X)\,\prob(\s t_{l-1}<|X|\le\s t_l) 
		+ \esp |X|^\g \, \prob(|X|>\s t_L ) - |\t_i|^\g,
	\end{align*}
	where $X\sim\calN(\t_i,\s^2)$. Using Lemma~\ref{lemma_CaiLowamlior} we get 
	\begin{align*}
		\Big|\sum_{l=0}^L \esp \hat{P}_{\g,K_l,M_l}(X)\,\prob(\s t_{l-1}<|X|\le\s t_l)\Big| &\le 
		\max_{l=0,\ldots,L} \big|\esp\hat{P}_{\g,K_l,M_l}(X)\big|\,\prob(|X|\le \s t_L) 
		\\ 
		&\le C\s^\g 6^{K_L} K_L^{1+\g/2} e^{c\t_i^2/(8\s^2)} e^{-\t_i^2/(8\s^2)} \phantom{\sum_{l=0}} 
		\\
		&\le C \s^\g (\log d)^{1+\g/2} \, 6^{9c\log d} \, e^{9(c-1)\log d} 
	\end{align*}
	and the last upper bound is smaller than $C\s^\g \log^{-\g/2}(s^2/d)$ if $c>0$ is small enough. On the other hand, using \eqref{eqproofo} we find that 
	{\color{black}
	\begin{align*}\label{ggg}
		\big|\esp |X|^\g \, \prob( |X|>\s t_L ) - |\t_i|^\g\big| &\le \big|\esp |X|^\g - |\t_i|^\g\big| + |\t_i|^\g \prob(|X|\le \s t_L) \\
		&\le C\left( \frac{\s^\g }{\log^{\g/2}(s^2/d)}\fcar_{\g\le 1}+
		\frac{\s |\t_i|^{\g-1}}{\sqrt{\log(s^2/d)}}\fcar_{\g>1} \right) + 2|\t_i|^\g e^{-\frac{\t_i^2}{8\s^2}} \nonumber \\
		&\le C\left(\frac{\s^{\g}}{\log^{\g/2}(s^2/d)} + \frac{\s |\t_i|^{\g-1}}{\sqrt{\log(s^2/d)}} \fcar_{\g>1}\right). \nonumber
	\end{align*}
	Finally, we get
	\begin{equation}\label{proof_eq6}
		\Big(\sum_{i\in S: |\t_i|>2\s t_L} B_i\Big)^2 \le C\left(\frac{\s^{2\g}s^2}{\log^\g(s^2/d)} + \frac{\s^2}{\log(s^2/d)}\Big( \sum_{i\in S} |\t_i|^{\g-1} \Big)^2\fcar_{\g>1}\right).
	\end{equation}	
}
\bigskip

$4^{\circ}.$ {\it Variance for $i\in S$.} We consider the same three cases as in item $3^{\circ}$ above. For the first two cases, it suffices to use a coarse bound granting that, for all $i\in S$, 
\begin{align} \label{rough}
		V_i \le \esp_\t [\xi_\g^2(y_{1,i}, y_{2,i})] = \sum_{l=0}^{L} \esp \hat{P}^2_{\g,K_l,M_l}(X) \,\prob(\s t_{l-1}<|X|\le\s t_l) + \esp |X|^{2\g}  \,\prob(|X|>\s t_L)
	\end{align}
	where $X\sim\calN(\t_i,\s^2)$.

\medskip

$4.1^{\circ}.$ {\it Case $0<|\t_i|<2\s t_0$.}
In this case, we deduce from \eqref{rough} that
	\begin{equation*}
		V_i \le \max_{l=0,\ldots,L} \esp \hat{P}^2_{\g,K_l,M_l}(X) + \esp |X|^{2\g},
	\end{equation*}
	where $X\sim\calN(\t_i,\s^2)$. Lemma~\ref{lemma_CaiLow2} and the fact that $\esp |X|^{2\g}\le C(\s^{2\g} + \s^2|\t_i|^{2\g-2} + |\t_i|^{2\g})$ (cf. Lemma~\ref{lemma_pluggedestimator}) imply
	\begin{align*}
		V_i &\le C (M_L^{2\g} 2^{8K_L} + \s^{2\g} + |\t_i|^{2\g}) \\
		&\le C(\s^{2\g} \log^\g(d) \, d^{72c\log 2} + \s^{2\g}\log^\g(s^2/d)).
	\end{align*}
	Hence, choosing $c>0$ small enough and using the assumption $s\ge 2\sqrt{d}$, we conclude that
	\begin{equation}\label{proof_eq7}
	s \max_{i:0<|\t_i|<2\s t_0} V_i \le \frac{C\s^{2\g}s^2}{\log^\g(s^2/d)}.
	\end{equation}	
\smallskip

$4.2^{\circ}.$ {\it Case $2\s t_0<|\t_i|\le2\s t_L$.}	
	 As in item $3.2^{\circ}$ above, we denote by $l_0\in\{1,\ldots,L-1\}$ the integer such that $\s t_{l_0} < |\t_i| \le \s t_{l_0+1}$. We deduce from \eqref{rough} that
	\begin{align*}
		V_i &\le \max_{l=0,\ldots,l_0-1} \esp \hat{P}^2_{\g,K_l,M_l}(X) \, \prob(|X|\le\s t_{l_0-1}) + \max_{l=l_0,\ldots,L} \esp \hat{P}^2_{\g,K_l,M_l}(X) + \esp|X|^{2\g},
	\end{align*}
	where $X\sim\calN(\t_i,\s^2)$.
	The second and third terms on the right hand side are controlled as in  item~$4.1^{\circ}$, with the only difference that now we have
	$
		\esp|X|^{2\g} \le C  (\s^{2\g} + |\t_i|^{2\g}) \le C \s^{2\g}\log^{\g} (d).
	$		
	For the first term, we find using Lemma~\ref{lemma_CaiLowamlior} that, for $X\sim\calN(\t_i,\s^2)$,
	\begin{align}\label{gg}
		&\max_{l=0,\ldots,l_0-1} \esp \hat{P}^2_{\g,K_l,M_l}(X) \, \prob(|X|\le\s t_{l_0-1}) \phantom{\Big(\Big)}\\
		&\qquad \le C\s^{2\g} \Big[ (\s/M_0)^{4-2\g} + (\s/M_{l_0-1})^{4-2\g}\Big] 6^{2K_{l_0-1}} e^{c\log(1+4/c)\t_i^2/(4\s^2)}\, e^{-\t_i^2/(8\s^2)} \phantom{\Big(\Big)} \nonumber \\
		&\qquad \le C\s^{2\g} \log^{\g}(d) e^{(c\log 6+4c\log(1+4/c)-1/2)t^2_{l_0-1}}. \phantom{\Big(\Big)}\nonumber 
	\end{align}
	Choosing $c>0$ small enough allows us to obtain the desired bound
	\begin{equation}\label{proof_eq8}
		s \max_{i: 2\s t_0<|\t_i|\le2\s t_L} V_i \le \frac{C\s^{2\g}s^2}{\log^\g(s^2/d)} .
	\end{equation}	

\smallskip

$4.3^{\circ}.$ {\it Case $|\t_i|>2\s t_L$.}
	We first note that, for $X\sim\calN(\t_i,\s^2)$:
\begin{align*}
	\var \big(|y_{1,i}|^\g \,\fcar_{|y_{2,i}|>\s t_L} \big)&=\prob(|X|>\s t_L)\big[\var (|X|^\g) + (\esp |X|^\g)^2 \prob(|X|\le\s t_L) \big]\\
	&\le C\big[ \s^2|\t_i|^{2\g-2} + |\t|_i^{2\g}\prob(|X|\le\s t_L)\big],
	\end{align*}
	{\color{black}
where we have used the inequalities $\var (|X|^\g)\le C\s^2|\t_i|^{2\g-2}$  and  $(\esp |X|^\g)^2\le \esp |X|^{2\g}\le C(\s^2|\t_i|^{2\g-2} + |\t_i|^{2\g})\le C|\t_i|^{2\g}$ that are valid due to Lemma~\ref{lemma_pluggedestimator} and to the fact that $|\t_i|>\s$. Thus, we obtain
	\begin{align*}
		V_i &\le 2 \var \Big(\sum_{l=0}^{L} \hat{P}_{\g,K_l,M_l}(y_{1,i}) \,\fcar_{\s t_{l-1}<|y_{2,i}|\le\s t_l} \Big)+2 \var \big(|y_{1,i}|^\g \,\fcar_{|y_{2,i}|>\s t_L} \big)\\
	&\le 2  \sum_{l=0}^{L} \esp \hat{P}^2_{\g,K_l,M_l}(X) \,\prob(\s t_{l-1}<|X|\le\s t_l) + C\big[ \s^2|\t_i|^{2\g-2} + |\t_i|^{2\g}\prob(|X|\le\s t_L)\big]	\\
	&\le C \Big(\max_{l=0,\ldots,L} \esp \hat{P}^2_{\g,K_l,M_l}(X) \, \prob(|X|\le\s t_L) 
		+ \s^{2\g} +  \s^2|\t_i|^{2\g-2}\fcar_{\g>1}+ |\t_i|^{2\g}\prob(|X|\le\s t_L) \Big) \phantom{\sum_L^L}
			\end{align*}
since for $0<\g\le 1$ we have $|\t_i|^{2\g-2}\le \s^{2\g-2} $ due to the fact that $|\t_i|>\s$. In the last display, the term
$
\max_{l=0,\ldots,L} \esp \hat{P}^2_{\g,K_l,M_l}(X) \, \prob(|X|\le\s t_L) 
$		
is controlled via an argument analogous to \eqref{gg} while 
$$
|\t_i|^{2\g}\prob(|X|\le\s t_L)\le |\t_i|^{2\g}\prob(|\xi|\ge |\t_i|/(2\s)) \le 2|\t_i|^{2\g} e^{-\frac{\t_i^2}{8\s^2}} 
		\le C\s^{2\g}, \ \ \xi\sim\calN(0,1),
$$ 
due to the fact that $t_L<|\t_i|/(2\s)$.
This allows us to conclude that
	\begin{eqnarray}\label{proof_eq9}
		\sum_{i\in S: |\t_i|>2\s t_L} V_i \le C\left(\frac{\s^{2\g}s^2}{\log^\g(s^2/d)} + \s^2 \sum_{i=1}^d|\t_i|^{2\g-2}\fcar_{\g>1}\right).
	\end{eqnarray}	
	}
	The result of the theorem follows now from~(\ref{proof_eq1}), (\ref{proof_eq2}), (\ref{proof_eq3}), (\ref{proof_eq4}), (\ref{proof_eq5}), (\ref{proof_eq6}), (\ref{proof_eq7}), (\ref{proof_eq8}), and (\ref{proof_eq9}). 
	
{\color{black}
\subsection{Proof of Theorem~\ref{theorem_upperbound_even}}
Set $\s_*=\varepsilon \sqrt{\g}$. Since $y_{i,m}$ are mutually independent with $ \esp_\t \Big[\prod_{m=1}^{\g}y_{i,m}\Big]=\t_i^\g$ we have
\begin{align*}\nonumber
\esp_\t \Big[\big(\prod_{m=1}^{\g}y_{i,m}-\t_i^\g \big)^2\Big] &= \esp_\t \Big[\prod_{m=1}^{\g}y_{i,m}^2\Big]-\t_i^{2\g} =(\t_i^2+\s_*^2)^\g -\t_i^{2\g}\\
&=\sum_{j=1}^\g {\g \choose j}\t_i^{2(\g-j)}\s_*^{2j} \le C(\s_*^2\t_i^{2(\g-1)}+ \s_*^{2\g}). \label{even1}
\end{align*}
The theorem follows from this inequality and the fact that
$$
\esp_\t \big[(\tilde{N}_\g -N_\g(\t))^2\big] = \esp_\t \Big[\big(\sum_{i=1}^d \big\{\prod_{m=1}^{\g}y_{i,m}-\t_i^\g\big\}\big)^2\Big] =
\sum_{i=1}^d  \esp_\t \Big[\big(\prod_{m=1}^{\g}y_{i,m}-\t_i^\g \big)^2\Big].
$$

}	

{\color{black}
\subsection{Proof of Theorem~\ref{corollary_upperbound1}}

{\it Proof of part (i).}  
Set
$\phi = \varepsilon^{\g}s\log^{-\g/2}(s^2/d)$.
First, assume that $\|\t\|_\g\ge\phi^{1/\g}$. Then, using the inequality $|a-b|\ |b|^{\g-1} \le |a^\g-b^\g|$, $\forall a,b>0$, and Theorem~\ref{theorem_upperbound1}  we get
\begin{align*}
	\esp_\t \big(\hat{n}^*_\g-\|\t\|_\g\big)^2 &\le \frac{\esp_\t \big(\hat{N}^*_\g-N_\g(\t)\big)^2}{\|\t\|_\g^{2\g-2}}
	\\
	 &\le  C\left(\phi^2 + 
		 \frac{  \varepsilon^2 s^{2/\g}}{\log(s^2/d)} \|\t\|_{\g}^{2\g-2}
		\right) (\|\t\|_\g^{2\g-2})^{-1}
		\\
	 &\le  C \left(\phi^{2/\g} + 
		 \frac{  \varepsilon^2 s^{2/\g}}{\log(s^2/d)}  \right)\le C \phi^{2/\g},
\end{align*}
which is the desired bound. Next, assume that $\|\t\|_\g<\phi^{1/\g}$. Using the inequality $|a-b|\le |a^\g-b^\g|^{1/\g}$, $\forall a,b>0$, Jensen's inequality, and Theorem~\ref{theorem_upperbound1}  we get
\begin{align*}
	\esp_\t \big[(\hat{n}^*_\g-\|\t\|_\g)^2\big] &\le \esp_\t \big[|\hat{N}^*_\g-N_\g(\t)|^{2/\g} \big]
	\\
	 &\le  C\left(\phi^2 + 
		 \frac{\varepsilon^2 s^{2/\g}}{\log(s^2/d)} \|\t\|_\g^{2\g-2}
		\right) ^{1/\g}
	\\
	 &\le  C\left(\phi^2 + 
		 \frac{\varepsilon^2 s^{2/\g}}{\log(s^2/d)} \phi^{2-2/\g}
		\right) ^{1/\g}
		\le C \phi^{2/\g}.	
\end{align*}
{\it Proof of part (ii).} We follow the same lines as the proof of part (i) but now we set $\phi=\varepsilon^{\g}d^{1/2}$. If $\|\t\|_\g\ge\phi^{1/\g}$ we use the inequality $|a-b|\ |b|^{\g-1} \le |a^\g-b^\g|$, $\forall a,b>0$, Theorem~\ref{theorem_upperbound_even} and the fact that $\|\t\|_{2\g-2}\le \|\t\|_\g$ for $\g\ge 2$ to obtain
\begin{align*}
	\esp_\t \big(\hat{n}^*_\g-\|\t\|_\g\big)^2 &\le \frac{\esp_\t \big(\hat{N}^*_\g-N_\g(\t)\big)^2}{\|\t\|_\g^{2\g-2}}
	\\
	 &\le  C\left(\phi^2 + \varepsilon^2 \|\t\|_{2\g-2}^{2\g-2}
		\right) (\|\t\|_\g^{2\g-2})^{-1}
		\le  C \left(\phi^{2/\g} + 
		 \varepsilon^2 \right)\le C \phi^{2/\g},
\end{align*}
which is the desired bound. On the other hand, if $\|\t\|_\g<\phi^{1/\g}$ the inequality $|a-b|\le |a^\g-b^\g|^{1/\g}$, $\forall a,b>0$, Jensen's inequality, Theorem~\ref{theorem_upperbound_even} and the fact that $\|\t\|_{2\g-2}\le \|\t\|_\g$ for $\g\ge 2$ yield
\begin{align*}
	\esp_\t \big[(\hat{n}^*_\g-\|\t\|_\g)^2\big] &\le \esp_\t \big[|\hat{N}^*_\g-N_\g(\t)|^{2/\g} \big]
	\\
	 &\le  C\left(\phi^2 + \varepsilon^2 \|\t\|_{2\g-2}^{2\g-2}
		\right)^{1/\g} \le C \phi^{2/\g}.
\end{align*}

}

\subsection{Proof of Theorem~\ref{theorem_upperbound2}}

Denoting by $S$ the support of $\t$ we have  
\begin{align*}
	\hat{N}_\g^* -N_\g(\t) &= \sum_{i\in S} \big\{ |y_i|^\g-\varepsilon^\g\a_\g-|\t_i|^\g\big\} - \sum_{i\in S} \big\{|y_i|^\g-\varepsilon^\g\a_\g\big\} \fcar_{y_i^2\le 2\varepsilon^2\log(1+d/s^2)} \\
	&+ \sum_{i\not\in S} \big\{|y_i|^\g-\varepsilon^\g\a_\g\big\} \fcar_{y_i^2> 2\varepsilon^2\log(1+d/s^2)},
\end{align*}
so that
\begin{align*}
	\esp_\t\big[\big(\hat{N}_\g^* -N_\g(\t)\big)^2 \big]&\le 4\esp_\t \Big[\Big( \sum_{i\in S} \big\{ |y_i|^\g-|\t_i|^\g\big\} \Big)^2 \Big]+  2^{\g+2}\varepsilon^{2\g}s^2\log^\g(1+d/s^2) \\ &+ 4\varepsilon^{2\g}s^2\a_\g^2 +  4d\varepsilon^{2\g}\esp\Big[\big(|\xi|^\g-\a_\g\big)^2\fcar_{\xi^2>2\log(1+d/s^2)}\Big]
\end{align*}
where $\xi\sim\nzeroun$. Using Lemma~\ref{lemma_pluggedestimator} we get
{\color{black}
\begin{align*}
	\esp_\t\Big[ \Big( \sum_{i\in S} \big\{ |y_i|^\g-|\t_i|^\g\big\} \Big)^2\Big] & =  \sum_{i\in S} \esp_\t\big[ ( |y_i|^\g-|\t_i|^\g)^2\big] +
	\sum_{i,j\in S, i\ne j} \big(\esp_\t |y_i|^\g-|\t_i|^\g\big) \big(\esp_\t  |y_j|^\g -|\t_j|^\g\big)
	\\
	&\le 	
	C\Big(\varepsilon^{2\g}s + \varepsilon^{4}\sum_{|\t_i|>\varepsilon} |\t_i|^{2\g-4} \Big) + C\Big(\varepsilon^{2\g}s^2
	+  \varepsilon^{4}\Big(\sum_{|\t_i|>\varepsilon} |\t_i|^{\g-2}\Big)^2\Big)
	\\
	&\le 	
	C\Big(\varepsilon^{2\g}s^2 + \varepsilon^{2}\sum_{i=1}^d |\t_i|^{2\g-2} \fcar_{\g\ge1} + 
	  \varepsilon^{2}\Big(\sum_{i=1}^d |\t_i|^{\g-1}\Big)^2\fcar_{\g\ge1}\Big)
	 \\
	&\le C\Big(\varepsilon^{2\g}s^2 + \varepsilon^2 s^{2/\g}\|\t\|_\g^{2\g-2} \fcar_{\g\ge1}\Big).
\end{align*}
where we have used \eqref{21} and \eqref{22}.
}
Next, we use the fact that, for $\xi\sim\nzeroun$ and any 
$x>0$, $a\ge0$, 
\begin{equation*}
	\esp \big( |\xi|^a \fcar_{|\xi| > x} \big) \le C x^{a-1} e^{-x^2/2},  \quad  \prob(|\xi|>x)\ge C(1+x)^{-1}e^{-x^2/2},
\end{equation*}
where $C$ depends only on $a$. Choosing $x= \sqrt{2\log(1+d/s^2)}\ge \sqrt{2\log(5)}$ (as $d\ge 4s^2$), we obtain 
\begin{equation*}
	\a_\g \le C \frac{x^{\g-1}e^{-x^2/2}}{x^{-1}e^{-x^2/2}} \le C \log^{\g/2}(1+d/s^2)
\end{equation*}
The same property implies that
\begin{align*}
	\esp\Big[\big(|\xi|^\g-\a_\g\big)^2\fcar_{\xi^2>2\log(1+d/s^2)}\Big] &\le 2\,\esp \big(|\xi|^{2\g}\fcar_{\xi^2>2\log(1+d/s^2)}\big) + 2\a_\g^2 \prob\big(\xi^2>2\log(1+d/s^2)\big) \phantom{a^2_1\Big(\Big)} \\
	&\le C \frac{s^2}{d} \log^{\g}(1+d/s^2).
\end{align*}
Combining the above inequalities proves the theorem.

{\color{black}\subsection{Proof of Theorem~\ref{corollary_upperbound2}}

We act as in the proof of Theorem~\ref{corollary_upperbound1} with suitable modifications.
Namely, set
$\phi = \varepsilon^{\g}s\log^{\g/2}(s^2/d)$.
If $\|\t\|_\g\ge\phi^{1/\g}$ then using Theorem~\ref{theorem_upperbound2}  we get
\begin{align*}
	\esp_\t \big(\hat{n}^*_\g-\|\t\|_\g\big)^2 &\le \frac{\esp_\t \big(\hat{N}^*_\g-N_\g(\t)\big)^2}{\|\t\|_\g^{2\g-2}}
	\le  C\left(\phi^2 + 
		 \varepsilon^2 s^{2/\g} \|\t\|_{\g}^{2\g-2}
		\right) (\|\t\|_\g^{2\g-2})^{-1}
		\\
	 &\le  C \left(\phi^{2/\g} + 
		  \varepsilon^2 s^{2/\g}  \right)\le C \phi^{2/\g}.
\end{align*}
On the other hand, if $\|\t\|_\g<\phi^{1/\g}$ then using  Theorem~\ref{theorem_upperbound2}  we get
\begin{align*}
	\esp_\t \big[(\hat{n}^*_\g-\|\t\|_\g)^2\big] &\le \esp_\t \big[|\hat{N}^*_\g-N_\g(\t)|^{2/\g} \big]
	\le  C\left(\phi^2 + 
		 \varepsilon^2 s^{2/\g} \|\t\|_\g^{2\g-2}
		\right) ^{1/\g}
	\\
	 &\le  C\left(\phi^2 + 
		 \varepsilon^2 s^{2/\g} \phi^{2-2/\g}
		\right) ^{1/\g}
		\le C \phi^{2/\g}.	
\end{align*}

}

\section{Lemmas for the proof of Theorem~\ref{theorem_upperbound1}}

\begin{lemma}\label{lemma_pluggedestimator}
	If $X\sim\calN(\vt,\s^2)$ and $\g>0$, then
		\begin{align*}
		\ |\esp(|X|^\g) - |\vt|^\g| &\le C\Big(\s^\g \fcar_{|\vt|\le\s}  +  \s^2|\vt|^{\g-2} \fcar_{|\vt|>\s}\Big), \\
		\ \var (|X|^\g) &\le C\Big(\s^{2\g} \fcar_{|\vt|\le\s}  +  \s^2|\vt|^{2\g-2} \fcar_{|\vt|>\s}\Big).
	\end{align*}
	\end{lemma}

\begin{proof}
	Set for brevity
	\begin{equation*}
		g(x)=|x|^\g, \quad b_\g = \esp(|X|^\g)- |\vt|^\g.
	\end{equation*}
	First, note that if $|\vt|\le\s$ we have $|b_\g| \le C\s^\g$. Now, consider the case $|\vt|>\s$. Then,
	\begin{align*}
		|b_\g| \le \frac{1}{\sqrt{2\pi}\s}\left[ \Big|\int_{|x|>|\vt|/2}(g(x+\vt)-g(\vt))e^{-\frac{x^2}{2\s^2}}dx\Big|+\Big|\int_{|x|\le |\vt|/2}(g(x+\vt)-g(\vt))e^{-\frac{x^2}{2\s^2}}dx\Big|\right].
	\end{align*}
We now bound separately the two terms on the right-hand side of this inequality. Using the second order Taylor expansion of $g$ around $\vt$ and the symmetry of the Gaussian distribution we get 
	\begin{align*}
		\Big|\int_{|x|\le |\vt|/2}(g(x+\vt)-g(\vt))e^{-\frac{x^2}{2\s^2}}dx\Big|&\le 
		\frac12 \int_{|x|\le |\vt|/2} \max_{|u|\le  |\vt|/2}|g''(\vt+u)|\,x^2 e^{-\frac{x^2}{2\s^2}}dx\\
		&\le C|\vt|^{\g-2} \int_{|x|\le |\vt|/2} x^2 e^{-\frac{x^2}{2\s^2}}dx \le C\s^3|\vt|^{\g-2}.
	\end{align*}
	On the other hand, the first integral in the bound for $|b_\g|$ is smaller than
	\begin{align*}
		C \int_{|x|>|\vt|/2} \big\{ |x|^\g+|\vt|^\g \big\} e^{-\frac{x^2}{2\s^2}}dx &\le C\s^{\g+1}\int_{|t|>|\vt|/(2\s)}|t|^\g e^{-\frac{t^2}{2}}dt + C\s|\vt|^\g \int_{|t|>|\vt|/(2\s)} e^{-\frac{t^2}{2}}dt  \\
		&\le C\s^{\g+1} \frac{|\vt|^{\g-1}}{\s^{\g-1}} e^{-\frac{\vt^2}{8\s^2}}  + C \s^2 |\vt|^{\g-1} e^{-\frac{\vt^2}{8\s^2}} \phantom{\int_{|t|>|\vt|/(2\s)} e^{-\frac{t^2}2}} \\
		&\le C \s^2 |\vt|^{\g-1} e^{-\frac{\vt^2}{8\s^2}}. 
	\end{align*}
	Combining the above inequalities yields the desired bound for the bias. The bound on the variance follows immediately since	\begin{align*}
		\var (|X|^\g) &= \esp(|X|^{2\g}) - \big(\esp |X|^\g\big)^2 
		= b_{2\g} + |\vt|^{2\g} - \big[ b_\g + |\vt|^\g \big]^2 
		\le  b_{2\g}. 
	\end{align*}
	
\end{proof}

\begin{lemma}\label{lemma_hermite2}
	Let $\vt\in \RR$ and $X\sim \calN(\vt,1)$. For any $k\in \mathbb N$, the $k$-th Hermite polynomial satisfies
	\begin{align*}
				\esp H_k(X) &= \vt^k, \\ 
			\esp H_k^2(X) &\le k^k(1+\vt^2/k)^k. 
	\end{align*}
\end{lemma}

	The proof of this lemma can be found in~\cite{CaiLow2011}.

\begin{lemma}\label{lemma_CaiLow1}
	Let $\hat{P}_{\g,K,M}$ be defined in~(\ref{parametres}) with parameters $K=K_l$ and $M=M_l$ for some $l\in\{0,\ldots,L\}$ and small enough $c>0$. If $X\sim\calN(0,\s^2)$, then
	\begin{equation*}
	\esp \hat{P}_{\g,K,M}^2(X) \le C \s^{2\g}\frac{6^{2K}}{(M/\s)^{4-2\g}},
	\end{equation*}
	where $C>0$ is a constant depending only on $\g$.
\end{lemma}

\begin{proof}
	Recall that, for the Hermite polynomials, $\esp (H_k(\xi)H_j(\xi))=0$ if $k\ne j$ and $\xi \sim\calN(0,1)$.
	Using this fact and then Lemmas~\ref{lemma_coefficientsgamma} and~\ref{lemma_hermite2} we obtain
	\begin{align*}
		\esp \hat{P}_{\g,K,M}^2(X) &= M^{2\g} \sum_{k=1}^K a^2_{\g,2k} (\s/M)^{4k} \esp H^2_{2k}(X/\s) \\
		&\le C 6^{2K} M^{2\g} \sum_{k=1}^K (2k)^{2k} (\s/M)^{4k}.
	\end{align*}
	Moreover, since $\s^2/M^2 = c/(8K)$ we have
	\begin{align}\label{splitt}
		\sum_{k=1}^K (2k)^{2k} (\s/M)^{4k} &\le \frac{4\s^4}{M^4} + \sum_{2\le k\le \log(M/\s)} (\s/M)^{4k} \big(2\log(M/\s)\big)^{2k}\\ 
		& \quad + \sum_{\log(M/\s)<k \le K} (c/4)^{2k} \le \frac{C\s^4}{M^4}\nonumber
	\end{align}
	if $c$ is small enough. The result follows.
\end{proof}

\begin{lemma}\label{lemma_CaiLow2}
	Let $\hat{P}_{\g,K,M}$ be defined in~(\ref{parametres}) with parameters $K=K_l$ and $M=M_l$ for some $l\in\{0,\ldots,L\}$ and small enough $c>0$. If $X\sim\calN(\vt,\s^2)$ with $|\vt|\le M$, then
	\begin{align*}
		 &\big| \esp \hat{P}_{\g,K,M}(X) - |\vt|^\g  \big| \le C \Big(\frac{M}{K}\Big)^\g, \\
		 &\esp \hat{P}^2_{\g,K,M}(X) \le C M^{2\g} 2^{8K},
			\end{align*}
	where $C>0$ is a constant  depending only on $\g$.
\end{lemma}

\begin{proof}
	To prove the first inequality of the lemma, it is enough to note that, due to Lemma~\ref{lemma_hermite2}, 
	\begin{equation}\label{biass}
		\esp \hat{P}_{\g,K,M}(X) = \sum_{k=1}^K a_{\g,2k} M^{\g-2k} \vt^{2k}
	\end{equation}
	and to apply Lemma~\ref{approx}.
		For the second inequality, we use the bound
	\begin{equation}\label{secmom}
		\esp \hat{P}^2_{\g,K,M}(X) \le M^{2\g} \bigg(\sum_{k=1}^K  \s^{2k} |a_{\g,2k}| M^{-2k} \sqrt{\esp H_{2k}^2(X/\s)} \bigg)^2.
	\end{equation}
	Thus Lemmas~\ref{lemma_coefficientsgamma} and~\ref{lemma_hermite2} together with the relations  
	$|\vt|\le M$ and $K=(c/8)M^2/\s^2$ imply that, for small enough $c>0$,
	\begin{align*}
		\esp \hat{P}^2_{\g,K,M}(X) \le CM^{2\g} 6^{2K} \bigg(\sum_{k=1}^K M^{-2k} (2M^2)^k \bigg)^2 \le C M^{2\g} 2^{8K}.
	\end{align*}
\end{proof}

\begin{lemma}\label{lemma_CaiLowamlior}
	Let $\hat{P}_{\g,K,M}$ be defined in~(\ref{parametres}) with parameters $K=K_l$ and $M=M_l$ for some $l\in\{0,\ldots,L\}$ and small enough $c>0$. If $X\sim\calN(\vt,\s^2)$ with $|\vt|>2\s t_l$, then
	\begin{align*}
		& \big| \esp \hat{P}_{\g,K,M}(X) \big| \le C\s^\g 6^{K} K^{1+\g/2} e^{c\vt^2/(8\s^2)}, \\
			& \esp \hat{P}^2_{\g,K,M}(X) \le C\s^{2\g} (\s/M)^{4-2\g} 6^{2K} e^{c\log(1+4/c)\vt^2/(4\s^2)},
	\end{align*}
	where $C>0$ is a constant  depending only on $\g$.
\end{lemma}

\begin{proof}
	To prove the first inequality of the lemma, we use \eqref{biass} and Lemma~\ref{lemma_coefficientsgamma} to obtain
	\begin{align*}
		\big| \esp \hat{P}_{\g,K,M}(X) \big| \le C M^\g K 6^{K} \Big(\frac{\vt^2}{M^2}\Big)^K.
	\end{align*}
	Recall that $M^2=8\s^2K/c$ and $|\vt|>M$ by assumption of the lemma. Thus,
	\begin{align*}
		M^\g K 6^{K} \Big(\frac{\vt^2}{M^2}\Big)^K \le C \s^\g K^{1+\g/2} 6^{K} e^{K\log(\vt^2/M^2)}
	\end{align*}
	and the result follows since $K\log(\vt^2/M^2)= cM^2/8 \s^2\log(\vt^2/M^2) \le c\vt^2/8\s^2$. 
	
	We now prove the  second  inequality of the lemma.  Using \eqref{secmom} and then Lemmas~\ref{lemma_coefficientsgamma} and~\ref{lemma_hermite2} we get
	\begin{align*}
		\esp \hat{P}^2_{\g,K,M}(X) 
		&\le C M^{2\g} 6^{2K} \bigg( \sum_{k=1}^K (\s/M)^{2k} (2k)^k \Big(1+\frac{\vt^2}{2\s^2k}\Big)^k \bigg)^2.
	\end{align*}
As $M^2=8\s^2K/c$ and $|\vt|>M$, we have
$$
\frac{\vt^2}{2\s^2k} \ge \frac{M^2}{2\s^2K} = \frac4{c}\ge 2 
$$	
for $c>0$ small enough. Using this remark and the fact that the function $x\to x^{-1}\log(1+x)$ is decreasing for $x\ge2$ we obtain
$$
k \log\Big(1+\frac{\vt^2}{2\s^2k}\Big)\le \frac{c\log(1+4/c)\vt^2}{8\s^2}.
$$
Therefore,
	\begin{align*}
		\esp \hat{P}^2_{\g,K,M}(X) &\le CM^{2\g} 6^{2K} e^{c\log(1+4/c)\vt^2/(4\s^2)}  \bigg( \sum_{k=1}^K (\s/M)^{2k} (2k)^k \bigg)^2.
	\end{align*}
	Finally, the result follows by noticing that, by an argument analogous to \eqref{splitt}, we have 
	$$
	\sum_{k=1}^K (\s/M)^{2k} (2k)^k \le \frac{C\s^2}{M^2}.
	$$
\end{proof}


\section{Some facts from approximation theory}\label{sec:approximation}

We start with a proposition relating moment matching to best polynomial  approximation. 
It is similar to several results used in the theory of  estimation of non-smooth functionals starting from \citet{LepskiNemirovskiSpokoiny1999}. There exist different techniques of proving such results for specific examples. Thus, the proof in \cite{LepskiNemirovskiSpokoiny1999} is based on Riesz representation of linear operators, while  \citet{WuYang2016} provide an  explicit construction using Lagrange interpolation.  Here, for completeness we give a short proof for a relatively general setting based on optimization arguments. 

 Let $f: [-1,1]\to \RR$ be a continuous even function. Consider the accuracy of best polynomial  approximation of $f$:
\begin{equation*}\label{deltak}
\d_{K}(f)= \inf_{G\in \mathcal{P}_K} \max_{x\in [-1,1]} \big| f(x)-G(x)\big|
\end{equation*}
where $\mathcal{P}_K$ is the class of all real polynomials of degree at most $K$. 
       
  \begin{proposition}\label{mesuresf} 
Let $f: [-1,1]\to \RR$ be a continuous even function.
For any even  integer $K\ge 1$, there exist two probability measures $\tilde{\mu}_0$ and $\tilde{\mu}_1$ on $[-1,1]$ such that 
\begin{itemize}
\item[(i)] $\tilde{\mu}_0$ and $\tilde{\mu}_1$ are symmetric about 0;
\item[(ii)]  $\int t^l \tilde{\mu}_0(dt)=\int t^l \tilde{\mu}_1(dt)$ for $l=0,1,\ldots, K$;
\item[(iii)]  $\int f(t) \tilde{\mu}_1(dt)-\int f(t) \tilde{\mu}_0(dt)=2 \d_{K}(f)$.
\end{itemize}
\end{proposition}

\begin{proof}
Denote by $P_{\sf sym}$ the set of all probability measures on $[-1,1]$ that are symmetric about~0, 
 and by $P_{2}$ be the set of all signed measures on $[-1,1]$ with total variation not greater than 2.  
For $K=2m$, we have 
\begin{align}\label{immed}
&\sup_{(\nu_0, \nu_1)\in P_{\sf sym}\times  P_{\sf sym} : \int t^{l}d\nu_0(t)=\int t^{l}d\nu_1(t), \ l =0,\ldots, K} \Big( \int_{-1}^1 f(x)d\nu_0(x)- \int_{-1}^1 f(x)d\nu_1(x)\Big)\\ \nonumber
&\qquad \qquad =\sup_{\mu\in P_{\sf 2}  :  \int t^{2l}d\mu(t)=0, \ l=0,\ldots,m} \ \int_{-1}^1 f(x)d\mu(x)\\ \nonumber
&\qquad \qquad= \sup_{\mu \in P_{\sf 2} } \inf_{\a\in \RR^{m+1}} \int_{-1}^1 \Big(f(x)-\sum_{l=0}^m \a_l x^{2l}\Big)d\mu(x)\\  \nonumber
&\qquad\qquad= \inf_{\a\in \RR^{m+1}}\sup_{\mu \in P_{\sf 2} } \int_{-1}^1 \Big(f(x)-\sum_{l=0}^m \a_l x^{2l}\Big)d\mu(x)\\
 \nonumber
&\qquad\qquad=2 \min_{\a\in \RR^{m+1}}\max_{x \in [-1,1]} \Big| f(x)-\sum_{l=0}^m \a_l x^{2l} \Big| = 2\delta_K(f),
\end{align}
where the third equality follows from Sion's minimax theorem, and the second equality uses the fact that $f$ is an even function, so that the maximum over $\mu\in P_2$ in the second line of~\eqref{immed} is equal to the maximum over symmetric $\mu\in P_2$ satisfying the same moment constraints. Let $(\nu_0^*,\nu_1^*)$ be the pair of probability measures attaining the maximum in the first line of \eqref{immed}. The proposition follows by setting $\tilde{\mu}_i=\nu_i^*$, $i=0,1$. 
\end{proof}
 
As an immediate corollary of Proposition~\ref{mesuresf} for $f(x)=|x|^\g$, we obtain the following result.
  \begin{lemma}\label{mesures} 
For any even  integer $K\ge 1$ and any $M>0$, $\g>0$, there exist two probability measures $\tilde{\mu}_0$ and $\tilde{\mu}_1$ on $[-M,M]$ such that 
\begin{itemize}
\item[(i)] $\tilde{\mu}_0$ and $\tilde{\mu}_1$ are symmetric about 0;
\item[(ii)]  $\int t^l \tilde{\mu}_0(dt)=\int t^l \tilde{\mu}_1(dt)$ for $l=0,1,\ldots, K$;
\item[(iii)]  $\int |t|^\g \tilde{\mu}_1(dt)-\int |t|^\g \tilde{\mu}_0(dt)=2 M^\g\d_{K,\g}$.
\end{itemize}
\end{lemma}

For the function $f(x)=|x|^\g$, the asymptotically exact behavior of the best polynomial approximation $\d_{K,\g}$ as $K\to\infty$ is well known, see, for example, \cite[Theorem 7.2.2]{Timan} implying the following lemma. 
\begin{lemma}\label{approx}
If $\g>0$ is not an even integer, then there exist positive constants $c_*$ and $C^*$ depending only on $\g$ such that 
 $$
 {c_*}K^{-\g} \le \delta_{K,\g}\le {C^*}K^{-\g},\quad \forall \ K\in {\mathbb N}.
 $$
 \end{lemma}
 
Finally, the next lemma provides a useful bound on the coefficients $a_{\g,2k}$ in the canonical representation of the polynomial of best approximation   
	\begin{equation}\label{decompositionPgamma}
		 P_{\g,K}(x) = \sum_{k=0}^K a_{\g,2k} x^{2k}, \quad x \in \RR.
	\end{equation}
		\begin{lemma}\label{lemma_coefficientsgamma}
		Let $P_{\g,K}(\cdot)$ be the polynomial of best approximation of degree $2K$ for $|x|^\g$ on $[-1,1]$. Then the coefficients $a_{\g,2k}$ 
		in~(\ref{decompositionPgamma}) satisfy
		\begin{equation*}
			|a_{\g,2k}| \le C 6^{K}, \quad k=0,\dots,K,
		\end{equation*}
		where $C>0$ is a constant depending only on $\g$. 
	\end{lemma}
This lemma is an immediate corollary of the following more general fact, which is a consequence of Szeg\"o's theorem on the minimal eigenvalue of a lacunary version of the Hilbert matrix.
\begin{proposition}\label{lemma_Szego}
		Let $P(x)=\sum_{k=0}^N a_{k} x^{k}$ be a polynomial such that $|P(x)|\le 1$ for all $x\in [-1,1]$. 
		Then there exists an absolute constant $C>0$ such that 
		\begin{equation*}
			|a_{k}| \le C  (\sqrt{2}+1)^{N}
		\end{equation*}
		for all $k\in\{0,\dots,N\}$.
			\end{proposition}
	\begin{proof}
	We have
	\begin{equation}\label{Szego1}
	\int_{-1}^1
	\Big(\sum_{k=0}^N a_{k} x^{k}\Big)^2 dx = 2\sum_{i, j=0}^N \frac{a_ia_j}{i+j+1}\fcar_{i+j \,\text{even}}.
	\end{equation}
	It is easy to see that the quadratic form in \eqref{Szego1} is positive definite for all $N$. Furthermore, 
	as shown by \citet{Szego1936}, the minimal eigenvalue $\lambda_{\min}(N)$ of this quadratic form  satisfies
	\begin{equation*}\label{Szego2}
	\lambda_{\min}(N) = 2^{9/4}\pi^{3/2}N^{1/2}(\sqrt{2}-1)^{2N+3}(1+o(1)) \quad \text{as} \ N\to \infty.
	\end{equation*}
	Therefore, there exists an absolute constant $C_0>0$ such that $\lambda_{\min}(N) \ge C_0(\sqrt{2}-1)^{2N}$ for all~$N$. This inequality and \eqref{Szego1} imply that 
	$$
	C_0(\sqrt{2}-1)^{2N} \sum_{k=0}^N a_k^2 \le 1   
	$$
	and hence $\max_{k=0,\dots,N}|a_k| \le C_0^{1/2}(\sqrt{2}-1)^{-N}$. 
	\end{proof}

\section{Construction of the priors for the proof of Theorem~\ref{theo_inf}}

 The proof of Theorem~\ref{theo_inf} will be based on Theorem 2.15 in \cite{Tsybakov2009}. It proceeds by bounding the minimax risk from below by the Bayes risk with  the prior measures on $\theta$ that we are going to define in this section.

In what follows we set 
\begin{equation}\label{M}
\L=\sqrt{\log \Big(\frac{s^2}{d}\Big)}, \quad M=\varepsilon \L,
\end{equation}
 and  we denote by $K$ the smallest even  integer such that  
\begin{equation}\label{K}  
K\ge\frac{3}{2} e \log \Big(\frac{s^2}{d}\Big)= \frac{3}{2} e\L^2.
\end{equation}
We will also write for brevity
$$B=B_0(s).$$
In what follows, unless stated otherwise, $\tilde{\mu}_0$ and $\tilde{\mu}_1$ are the probability measures satisfying Lemma~\ref{mesures} where $M$ is defined in \eqref{M} and $K$ is the smallest even integer for which \eqref{K} holds.  

For $i=0,1$, the probability measure $\mu_i$ is defined as the distribution of random vector  
$\theta\in \RR^d$ with components $\theta_j$ having the form $\theta_j = \epsilon_j \eta_j$, $j=1,\dots, d$, where $\epsilon_j$ is a Bernoulli random variable with $\prob(\epsilon_j=1)=s/(2d)$, $\eta_j$ is distributed according to $\tilde{\mu}_i$, and $(\epsilon_1,\dots,\epsilon_d,  \eta_1,\dots, \eta_d)$ are mutually independent. 

Let $\PP_0$ and $\PP_1$ be the mixture probability measures defined by the relation
$$\PP_{i}(A)= \int_{\RR^d} \prob_\t(A)\; \mu_{i}(d\t), \  i = 0,1, $$
for any measurable set $A$.
The densities of $\PP_0$ and $\PP_1$ with respect to the Lebesgue measure on $\RR^d$ have the form
$$f_0(x)=\prod_{i=1}^d h(x_i) \quad \text{ and }\quad f_1(x)=\prod_{i=1}^d g(x_i),\quad x=(x_1,\ldots, x_d)\in \RR^d,
$$
respectively, where for  $x\in \RR$ we set
$$ h(x)=\frac{s}{2d} \varphi_{0}(x)+ \Big(1-\frac{s}{2d}\Big)\varphi(x)$$
and 
$$g(x)=\frac{s}{2d} \varphi_{1}(x)+ \Big(1-\frac{s}{2d}\Big)\varphi(x)$$ 
 with 
\begin{equation}\label{phii}  \varphi_i(x)= \int_\RR \varphi(x-t) \tilde{\mu}_i (dt), \  i = 0,1,
\end{equation}
where we denote by $\varphi(\cdot)$ the density of the $\calN (0,\varepsilon^2)$ distribution.

Note that  the measures $\mu_0$ and $\mu_1$ are not supported in $B$. We associate to them two probability 
  measures $\mu_{0,B}$ and $\mu_{1,B}$ supported in $B$ and the corresponding mixture measures defined by 
$$\mu_{i,B}(A)= \frac{\mu_i(A\cap B)}{\mu_i(B)},  \quad \PP_{i,B}(A)= \int_{\RR^d} \prob_\t(A)\; \mu_{i,B}(d\t),\quad i = 0,1,$$
for any measurable set $A$. 

\

\section{Proof of Theorem~\ref{theo_inf}}

Since we have $\ell(t)\ge \ell(a)\fcar_{t> a}$ for any $a>0$, it is enough to prove the theorem for the indicator loss $\ell(t)=\fcar_{t> a}$.
Introduce the following notation: 
$$m_{i}= \int_{\RR^d} N_\g(\t) \mu_{i}(d\t),\quad  \quad v_{i}^2=  \int_{\RR^d} (N_\g(\t)-m_i)^2 \mu_{i}(d\t),\quad i = 0,1.
$$
Note that Lemmas~\ref{mesures} and~\ref{approx} imply:
\begin{align}\label{mi}
m_{1}-m_{0} &= d\Big(\int_{\RR^d} |\t_1|^\g \mu_{1}(d\t)- \int_{\RR^d} |\t_1|^\g \mu_{0}(d\t)\Big)
= \frac{s}{2}\Big(\int_{-M}^M |t|^\g \tilde\mu_{1}(dt)- \int_{-M}^M |t|^\g \tilde\mu_{0}(dt)\Big)
\\ \nonumber
&= s M^{\g}\delta_{K,\g} \ge c_*  s(M/K)^\g \ge C_1\frac{ \varepsilon^\g s}{\L^\g},
\end{align}
where $C_1>0$ is a constant depending only on $\g$.

Let $V(P,Q)$ denote the total  variation distance between two probability measures $P$ and~$Q$. 
For any $u>0$ and any $c\in \RR$ we have, using Theorem 2.15 in \cite{Tsybakov2009},
\begin{equation}\label{fuzzy}
\inf_{\hat{T}} \sup_{\t \in B_0(s)} \prob_\t (|\hat{T}-N_\g(\t)|\ge u) \ge \frac{1-V'}{2},
\end{equation}
where
$$
V' = V(\PP_{0,B},\PP_{1,B})+\mu_{0,B}(N_\g(\t)\ge c  )+\mu_{1,B}(N_\g(\t)\le c+2u).
$$
We now apply \eqref{fuzzy} with the parameters
 $$c=m_{0}+3v_{0}, \quad u= \frac{m_{1}-m_{0}}{4}.$$ 
 By Chebyshev-Cantelli inequality,
\begin{equation}\label{beta0}
\mu_{0}(N_\g(\t)\ge c)\le \frac{v_0^2}{v_0^2 + (c-m_0)^2}=\frac{1}{10}.
\end{equation}
Next, since the measures $\tilde\mu_{0}$ and $\tilde\mu_{0}$ are supported in $[-M,M]$,
\begin{equation}\label{voi}\max(v_0^2,v_{1}^2) \le d M^{2\g}= d \varepsilon^{2\g} \L^{2\g}.
\end{equation} 
Thus, we may write
$$\max(v_0,v_{1})\le \Big( \frac{\sqrt{d}}{s}   \L^{2\g}\Big) \frac{\varepsilon^{\g} s}{\L^\g} ,$$
where, for $\bar C$ large enough, $\frac{\sqrt{d}}{s}   \L^{2\g}= \frac{\sqrt{d}}{s}  \log^\g (\frac{s^2}{d})\le C_1/12$ (recall that $s^2\ge \bar C d$ by assumption). Therefore,
\begin{equation}\label{vi}
\max(v_0,v_{1})\le \frac{C_1}{12} \frac{\varepsilon^\g s}{\L^\g}.
\end{equation}
It follows from  \eqref{mi}, \eqref{vi} and Chebyshev-Cantelli inequality that 
\begin{align}\label{beta1}
\mu_1(N_\g(\t)\le c+2u) &= \mu_1(N_\g(\t)-m_1\le -\frac{m_1+m_0}{2}+3v_0)\\ \nonumber
&\le \mu_1(N_\g(\t)-m_1\le -\frac{m_1- m_0}{2}+3v_0)\\ 
&\le  \mu_1\Big(N_\g(\t)-m_1\le -\frac{C_1}{4} \frac{\varepsilon^\g s}{\L^\g}\Big)\le \frac1{10}.\nonumber
\end{align}
By Lemma~\ref{moy}, we have $\mu_i(B)\ge 7/8$, $i=0,1$. Combining these inequalities with \eqref{beta0} 
and~\eqref{beta1} we immediately conclude that
\begin{equation}\label{betax}
\mu_{0,B}(N_{\g}(\t)\ge c)+\mu_{1,B}(N_{\g}(\t)\le c+2u)\le 8/35.
\end{equation}
Next, we consider the total variation distance $V(\PP_{0,B}, \PP_{1,B})$. 
Using Lemma~\ref{moy} we get that, for $\bar C$ large enough,
\begin{align}\label{V}
V(\PP_{0,B}, \PP_{1,B}) &\le V(\PP_{0,B}, \PP_0) +V(\PP_{0}, \PP_1) + V(\PP_1, \PP_{1,B}) \\
&\le  V(\PP_0, \PP_1) + \mu_0( B^c) + \mu_1(B^c) 
\nonumber \\
&\le  V(\PP_0, \PP_1) + 1/4
\nonumber \\
& \le \sqrt{\chi^2(\PP_1, \PP_0)/2} + 1/4
\nonumber \\
 &\le (\sqrt{2}+ 1)/4,
 \nonumber
\end{align}
where the last two inequalities are due to Pinsker's inequality and  Lemma~\ref{chi}, respectively. Combining \eqref{fuzzy},  \eqref{betax} and \eqref{V} we get that, if $s^2\ge \bar C d$ for $\bar C>0$ large enough, there exists a constant $C>0$ depending only on $\g$ such that 
$$\inf_{\hat{T}} \sup_{\t \in B_0(s)} \prob_\t \Big(|\hat{T}-N_\g(\t)|\ge C \frac{ \varepsilon^\g s}{\L^\g}\Big)> \frac1{16} .$$
This completes the proof.

\section{Lemmas for the proof of Theorem~\ref{theo_inf}}

\begin{lemma}\label{moy} For $i=0,1$, we have 
\begin{equation}\label{moy1} 
V(\PP_i, \PP_{i,B}) \le  \mu_i( B^c).
\end{equation}
Furthermore, there exists an absolute constant $\bar C>0$ such that, for any  $s^2\ge \bar C d$, 
\begin{equation}\label{moy2}
\mu_i(B^c)\le 1/8, \quad  i=0,1. 
\end{equation}
\end{lemma}
\begin{proof}
We can use, for example, Lemma~4 in \cite{CCT2018}. Repeating its argument we get that  
$V(\PP_i, \PP_{i,B}) \le  \mu_i( B^c) = \prob\big({\mathcal B}\big(d,\frac{s}{2d}\big)>s\big)\le e^{-\frac{s}{16}}$ where ${\mathcal B}\big(d,\frac{s}{2d}\big)$ is the binomial random variable with parameters $d$ and $\frac{s}{2d}$.
\end{proof}

\begin{lemma}\label{chi1}
Let $\tilde{\mu}_0$ and $\tilde{\mu}_1$ be two probability measures on $[-M,M]$ satisfying the moment matching property (ii) of Lemma~\ref{mesures} with some $K\ge 1$. Let $\varphi_0$ and  $\varphi_1$ be defined in \eqref{phii} where  $\varphi$ is the density of $\calN(0,\varepsilon^2)$ distribution. Then
$$
\int \frac{(\varphi_0(x) -\varphi_1(x))^2}{\varphi(x)} dx \le  \sum_{k=K+1}^\infty \frac{\L^{2k}}{k!}
$$
where $\L=M/\varepsilon$.
\end{lemma}
\begin{proof} By rescaling, it suffices to consider the case $\varepsilon=1$, $M=\L$. 
Introducing the notation ${\mathbb E}_i(k)=\int t^k \tilde{\mu}_i(dt)$, $i=0,1$, it is straightforward to check that
\begin{align*}
\int \frac{(\varphi_0(x) -\varphi_1(x))^2}{\varphi (x)} dx&= \int e^{\vt\vt'} \tilde{\mu}_1(d\vt)\tilde{\mu}_1(d\vt') + \int e^{\vt\vt'} \tilde{\mu}_0(d\vt)\tilde{\mu}_0(d\vt') - 2\int e^{\vt\vt'} \tilde{\mu}_1(d\vt)\tilde{\mu}_0(d\vt')\\
&= \sum_{k=0}^\infty \frac{1}{k!}\Big(({\mathbb E}_1(k))^{2} + ({\mathbb E}_0(k))^{2} - 2{\mathbb E}_1(k) {\mathbb E}_0(k)  \Big)\\
&= \sum_{k=0}^\infty \frac{1}{k!}\Big({\mathbb E}_1(k)  - {\mathbb E}_0(k)  \Big)^2.
\end{align*}
It remains to notice that ${\mathbb E}_1(k)  = {\mathbb E}_0(k)$ for $k=0,\dots, K$, by property (ii) of Lemma~\ref{mesures}, and 
$|{\mathbb E}_1(k)  - {\mathbb E}_0(k)|\le \L^{2k}$ for all $k$.
\end{proof}

\begin{lemma}\label{chi}
If  $s^2\ge 4d$, then
$$\chi^2(\PP_1, \PP_0) <  1/4.$$ 
\end{lemma}
\begin{proof}
Since $\PP_0$ and $\PP_1$ are product measures we have   
$$
\chi^2(\PP_1,\PP_0)=\left(1+ \int \frac{(g-h)^2}{h} \right)^d-1,
$$
cf., e.g., \cite[page 86]{Tsybakov2009}. It follows from the definition of $g$ and $h$ and from Lemma~\ref{chi1} that
\begin{align*}
\int \frac{(g-h)^2}{h} &\le \frac{1}{1-\frac{s}{2d} }\Big(\frac{s}{2d} \Big)^2\int \frac{(\varphi_{1}-\varphi_{0})^2}{\varphi}
\le 2 \Big(\frac{s}{2d} \Big)^2  \sum_{k=K+1}^\infty \frac{\L^{2k}}{k!}.
\end{align*}
Using the inequalities $k!\ge ({k}/{e})^k$ and $1+x\le e^x$  we get 
$$
\chi^2(\PP_1, \PP_0)\le \exp\Big( \frac{s^2}{2d}\sum_{k=K+1}^\infty
 \Big(\frac{e\L^{2}}{k}\Big)^k \Big)-1.
 $$
 Recall that $K\ge 3e\L^2/2$ and $K-2< 3e\L^2/2$.  Thus,
 \begin{align*}
\frac{s^2}{2d}\sum_{k=K+1}^\infty
 \Big(\frac{e\L^{2}}{k} \Big)^k\le \frac{s^2}{2d}\sum_{k=K+1}^\infty
 ({2}/{3})^k = \frac{s^2}{d}
 ({2}/{3})^K 
 <  \frac{4s^2}{9d} \exp\Big(3e \log(2/3)L^2/2\Big)= \frac{4}{9}\Big(\frac{s^2}{d}\Big)^a
 \end{align*}
where $a=1+3e \log(2/3)/2 < -0.6$. Since $s^2\ge 4d$ we get
 $ 
\chi^2(\PP_1, \PP_0)\le \exp(4^{0.4}/9)-1< 1/4.
 $
\end{proof}

{\color{black}

\section{Proof of Theorems~\ref{theorem_lowerbound_norm1} and~\ref{theorem_lowerbound_norm2}}

Theorems~\ref{theorem_lowerbound_norm1} and~\ref{theorem_lowerbound_norm2} are obtained as corollaries of 
Theorem~\ref{theo_inf_sparse} thanks to the following lemma.

\begin{lemma}\label{lemma_reduction}
Let $\g>1$. Then, for any $\phi>0$, any $\t\in \RR^d$ such that $\|\t\|_\g\le \phi^{1/\g}$, and any estimator $\hat T\ge 0$,
$$
 \prob_\t\, \Big( |\hat{T}-\|\t\|_\g|\ge \phi^{1/\g}\Big) \ge 
 \prob_\t\, \Big( |\hat{T}^\g-\|\t\|_\g^{\g}|\ge C\phi\Big). 
$$
\end{lemma}
\begin{proof}
Since $\|\t\|_\g\le \phi^{1/\g}$ we have
\begin{align*}
|\hat{T}-\|\t\|_\g|& = |\hat{T}-\|\t\|_\g| \fcar_{\{\hat{T} > 2\phi^{1/\g}\}}+ |\hat{T}-\|\t\|_\g| \fcar_{\{\hat{T} \le 2\phi^{1/\g}\}}
\\
& \ge \phi^{1/\g}\fcar_{\{\hat{T} > 2\phi^{1/\g}\}}+ |\hat{T}-\|\t\|_\g| \fcar_{\{\hat{T} \le 2\phi^{1/\g}\}}
\\
& \ge \phi^{1/\g}\fcar_{\{\hat{T} > 2\phi^{1/\g}\}}+ \frac{|\hat{T}^\g-\|\t\|_\g^\g|}{\g \max(\hat{T}^{(\g-1)/\g},\|\t\|_\g^{\g-1})} \fcar_{\{\hat{T} \le 2\phi^{1/\g}\}}
\\
& \ge \phi^{1/\g}\fcar_{\{\hat{T} > 2\phi^{1/\g}\}}+ \frac{|\hat{T}^\g-\|\t\|_\g^\g|}{2^{\g-1}\g \phi^{(\g-1)/\g}} \fcar_{\{\hat{T} \le 2\phi^{1/\g}\}}
\end{align*}
where we have used the inequality $|x^\g-y^\g|\le \g \max(x^{\g-1},y^{\g-1})|x-y|$, $\forall x,y >0$. This yields the result of the
lemma with $C=2^{\g-1}\g$.
\end{proof}

It suffices to prove  Theorems~\ref{theorem_lowerbound_norm1} and~\ref{theorem_lowerbound_norm2} for the indicator loss $\ell(t)=\fcar_{t\ge a}, a>0$,
and to consider the infimum only over non-negative estimators $\hat{T}\ge 0$ since the estimated functional is non-negative.  It follows from Lemma~\ref{lemma_reduction} that
$$
 \inf_{\hat{T}\ge 0} \sup_{\t\in B} \prob_\t\, \Big( |\hat{T}-\|\t\|_\g|\ge \phi^{1/\g}\Big) \ge 
 \inf_{\hat{T}'\ge 0} \sup_{\t\in B} \prob_\t\, \Big( |\hat{T}'-\|\t\|_\g^{\g}|\ge C\phi\Big),
$$ 
where $B=B_0(s)\cap \{\|\t\|_\g\le \phi^{1/\g}\}$.
 The result of Theorem~\ref{theorem_lowerbound_norm1}  follows immediately from this inequality with $\phi=c\varepsilon^\g s \log^{\g/2}(1+d/s^2)$  and
Theorem~\ref{theo_inf_sparse}.
Here, $c$ is a sufficiently small positive number. To prove Theorem~\ref{theorem_lowerbound_norm2}, it suffices to apply Theorem~\ref{theo_inf_sparse} with 
$s$ being the minimal integer greater than or equal to $\sqrt{ d}$ and to use the fact that the classes $B_0(s)$ are nested.

\section{Proof of Theorem~\ref{theorem_lowerbound_norm_hard} }

The proof is analogous to that of Theorem~\ref{theo_inf} subject to a modification that we detail here. 
Let $c$ and $u$ be as in the proof of Theorem~\ref{theo_inf}:  
$$c=m_{0}+3v_{0}, \quad u= \frac{m_{1}-m_{0}}{4}.$$ 
Define
$$
c'=c^{1/\g}, \quad u' = \frac{(c+2u)^{1/\g}-c^{1/\g}}{2}\,.
$$
Analogously to \eqref{fuzzy}, we obtain from Theorem 2.15 in \cite{Tsybakov2009} that
\begin{equation}\label{fuzzy1}
\inf_{\hat{T}} \sup_{\t \in B_0(s)} \prob_\t (|\hat{T}-\|\t\|_\g|\ge u') \ge \frac{1-V'}{2},
\end{equation}
where
$$
V' = V(\PP_{0,B},\PP_{1,B})+\mu_{0,B}(\|\t\|_\g\ge c'  )+\mu_{1,B}(\|\t\|_\g\le c'+2u').
$$
Note that this value is equal to $V'$ defined in the proof of Theorem~\ref{theo_inf}. Hence, $V'$ is bounded from above exactly as in the proof of Theorem~\ref{theo_inf} and to complete the proof of Theorem~\ref{theorem_lowerbound_norm_hard} we only need to check that $u' \ge {C\varepsilon s^{1/\g}}{\log^{1/2-\g}(s^2/d)}$, which is the desired rate.  Using the inequality $|x^\g-y^\g|\le \g \max(x^{\g-1},y^{\g-1})|x-y|$, $\forall x,y >0$, we get
$$
u'\ge \frac{2u}{\g(c+2u)^{(\g-1)/\g}}.
$$
 Next, due to \eqref{mi}, \eqref{voi} and the assumption that $s\ge 2 \sqrt{d}$, we have
 $$
 c+2u=\frac{m_{1}+m_{0}}{2} +3v_0 \le sM^\g + 3\sqrt{d} M^\g \le 3sM^\g = 3s\varepsilon^\g\Lambda^\g. 
 $$
Moreover, \eqref{mi} implies that $u\ge (C_1/4)s\varepsilon^\g \Lambda^{-\g}$. Thus, 
$u'\ge C \varepsilon s^{1/\g} \Lambda^{1-2\g}$ and we conclude by recalling the definition of $\Lambda$.

}
 
  \section*{Acknowledgements}

The work of Olivier Collier has been conducted as part of the project Labex MME-DII (ANR11-LBX-0023-01). The work of A.B.Tsybakov was supported by GENES and by the French National Research Agency (ANR) under the grants IPANEMA (ANR-13-BSH1-0004-02) and Labex Ecodec (ANR-11-LABEX-0047).


\begin{thebibliography}{100}
	{\small
		
				
		\bibitem[Cai and Low(2011)]{CaiLow2011}
		\textsc{T.T. Cai and M.G. Low}.
		\newblock {Testing composite hypotheses, Hermite polynomials and optimal estimation of a nonsmooth functional}.
		\newblock \emph {Annals of Statistics}, 39, 1012 -- 1041, 2011.
		
		\bibitem[Carpentier and Verzelen(2017)]{CarpentierVerzelen2017}
		\textsc{A. Carpentier and N. Verzelen}. 
		\newblock {Adaptive estimation of the sparsity in the Gaussian vector model}. 
		\newblock {arXiv:1703.00167}, 2017.
		
		\bibitem[Collier et al.(2017)]{CCT2017}
		\textsc{O. Collier, L. Comminges and A.B. Tsybakov}.
		\newblock {Minimax estimation of linear and quadratic functionals under sparsity constraints}.
		\newblock \emph{Annals of Statistics}, 45, 923 -- 958, 2016.

		\bibitem[Collier et al.(2018)]{CCTV2018}
		\textsc{O. Collier, L. Comminges, A.B. Tsybakov and N. Verzelen}.
		\newblock {Optimal adaptive estimation of linear functionals under sparsity}.
		\newblock \emph{Annals of Statistics}, 46(6A), 3130--3150, 2018.
		
		\bibitem[Comminges et al.(2018)]{CCT2018}
		\textsc{L. Comminges, O. Collier,  M.Ndaoud and A.B. Tsybakov}.
		\newblock {Adaptive robust estimation in sparse vector model}.
		\newblock {arXiv:1802.04230}, 2018.	
				
		
		\bibitem[Fukuchi and Sakuma(2017)]{FukuchiSakuma2017}
		\textsc{K. Fukuchi and J. Sakuma}.
		\newblock {Minimax optimal estimators for additive scalar functionals of discrete distributions}.
		\newblock {arXiv:1701.06381}, 2017.
		
		
		\bibitem[Jiao et al. (2015)]{JiaoVenkatHanWeissman2015}
		\textsc{J. Jiao, K. Venkat, Y. Han and T. Weissman}. 
		\newblock {Minimax estimation of functionals of discrete distributions}.  
		\newblock  \emph{IEEE Transactions on Information Theory}, 61, 2835 -- 2885, 2015.
	
		
		\bibitem[Han et al. (2017a)]{HanJiaoMukherjeeWeissman2017}
		\textsc{Y. Han, J. Jiao, R. Mukherjee and T. Weissman}.
		\newblock {On estimation of $L_r$-norms in Gaussian white noise models}.
		\newblock {arXiv:1710.03863}, 2017.
		
		\bibitem[Han et al. (2017b)]{HanJiaoWeissmanWu2017}
		\textsc{Y. Han, J. Jiao, T. Weissman and Y. Wu}.
		\newblock {Optimal rates of entropy estimation over Lipschitz balls}.
		\newblock {arXiv:1711.02141}, 2017.
			
		
		\bibitem[Lepski et al.(1999)]{LepskiNemirovskiSpokoiny1999}
		\textsc{O. Lepski, A. Nemirovski and V. Spokoiny}.
		\newblock {On estimation of the $L_r$ norm of a regression function}.
		\newblock \emph {Probability Theory and Relatated Fields}, 113, 221 -- 253, 1999.
		
		\bibitem[Nemirovski (2000)]{Nemirovski2000} 
		\textsc{A. Nemirovski}.
		\newblock{\it Topics in Non-parametric Statistics.} 
                 \newblock Ecole d'Et\'e de Probabilit\'es de Saint-Flour XXVIII - 1998. Lecture Notes in Mathematics, v. 1738. Springer, 2000.
		
		\bibitem[Szeg{\"o}(1936)]{Szego1936}
		\textsc{G.  Szeg{\"o}}.  
		\newblock {On  some  Hermitian  forms  associated  with  two  given  curves  of  the
                complex plane}.
                 \newblock \emph {Transactions Amer. Math. Society}, 40, 450 -- 461, 1936. 
                 		
                 \bibitem[Timan(1963)]{Timan}
		\textsc{A.F. Timan.}
		\newblock \emph{Theory of Approximation of Functions of a Real Variable}.
		\newblock {Pergamon Press}, 1963.

		
		\bibitem[Tsybakov(2009)]{Tsybakov2009}
		\textsc{A.B. Tsybakov.}
		\newblock \emph{Introduction to Nonparametric Estimation}.
		\newblock {Springer, New York}, 2009.
		
		\bibitem[Wu and Yang(2015)]{WuYang2015}  
		\textsc{Y.Wu and P.Yang.} 
		\newblock {Chebyshev polynomials, moment matching, and optimal estimation of the unseen}.  
		\newblock {arXiv:1504.01227}, 2015.
		
		\bibitem[Wu and Yang(2016)]{WuYang2016}  
		\textsc{Y.Wu and P.Yang.} 
		\newblock Minimax rates of entropy estimation on large alphabets via best polynomial approximation.
		 \newblock \emph{IEEE Transactions on Information Theory}, 62, 3702 -- 3720, 2016.

		
	}
\end{thebibliography}

%

  \end{document}